\documentclass[12pt,reqno]{amsart}
\usepackage{hyperref}
\usepackage{stmaryrd}
\usepackage{color, bm, amscd, tikz-cd}
\newcommand{\cB}{{\mathcal B}}

\newcommand{\cH}{{\mathcal H}}

\newcommand{\cK}{{\mathcal K}}

\newcommand{\cT}{{\mathcal T}}

\newtheorem{thm}{Theorem}[section]

\newtheorem{corollary}[thm]{Corollary}
\newtheorem{lemma}[thm]{Lemma}

\newtheorem{proposition}[thm]{Proposition}

\theoremstyle{definition}

\newtheorem{definition}[thm]{Definition}
\newtheorem{remark}[thm]{Remark}

\newtheorem{example}[thm]{Example}
\newtheorem{notation}[thm]{Notation}

\numberwithin{equation}{section}

\def\textmatrix#1&#2\\#3&#4\\{\bigl({#1 \atop #3}\ {#2 \atop #4}\bigr)}
\def\dispmatrix#1&#2\\#3&#4\\{\left({#1 \atop #3}\ {#2 \atop #4}\right)}
\numberwithin{equation}{section}

\def\textmatrix#1&#2\\#3&#4\\{\bigl({#1 \atop #3}\ {#2 \atop #4}\bigr)}
\def\dispmatrix#1&#2\\#3&#4\\{\left({#1 \atop #3}\ {#2 \atop #4}\right)}

\def\pbar{{\overline{p}}}

\begin{document}
\title[Toeplitz operators]{Algebraic Properties of Toeplitz operators on the symmetrized polydisk}
\author[B. Krishna Das]{B. Krishna Das}
\address{Department of Mathematics, Indian Institute of Technology Bombay, Powai, Mumbai, 400076, India.\\ dasb@math.iitb.ac.in, bata436@gmail.com}
\author[H. Sau]{Haripada Sau}
\address{Department of Mathematics, Indian Institute of Science Education and Research, Dr Homi Bhabha Rd, Pashan, Pune, Maharashtra, 411008, India.\\hsau@iiserpune.ac.in, haripadasau215@gmail.com}

\begin{abstract}
This paper is an effort to continue the legacy of the classically successful theory of Toeplitz operators on the Hardy space over the unit disk to a new domain in $\mathbb C^d$ -- the symmetrized polydisk. We obtain algebraic characterizations of Toeplitz operators, analytic Toeplitz operators, compact perturbation of Toeplitz operators and dual Toeplitz operators. We then revisit the operator theory of this domain considered first in \cite{SS}, to study the generalized Toeplitz operators and find a commutant lifting type result. This is a continuation of a previous work done in \cite{B-D-S}.
\end{abstract}
\maketitle

\section{Introduction}

Let $\mathbb D$ be the open unit disk and $\mathbb T$ be the unit circle in the complex plane $\mathbb C$, while $\mathbb D^d$ and $\mathbb T^d$ denote the corresponding cartesian products. The Hardy space over the unit disk, denoted by $H^2$, is the Hilbert space of functions that are analytic in $\mathbb D$ and the co-efficients of the Taylor series around the origin are square summable (making it naturally isomorphic to $\ell^2$, the space of square summable sequences in $\mathbb C$). A function $f$ analytic in $\mathbb D$ is in $H^2$ if and only if
\begin{align*}
\|f\|_{H^2}^2=\sup_{0<r<1}\int_{0}^{2\pi}|f(re^{i\theta})|^2d\theta <\infty.
\end{align*}Multiplication by the coordinate function `$z$', denoted by $T_z:f\mapsto zf(z)$, is the first example of a bounded operator on $H^2$ that comes to one's mind. This is an example of a fairly large class of operators that has remained a field of extensive research and development since mid 20th century, viz., the Toeplitz operators. For a {\em symbol} $\varphi$ in $L^\infty$, the algebra of essentially bounded function on $\mathbb T$ (the measures in this paper will be the Lebesgue measure unless otherwise specified), the {\em Toeplitz operator} $T_\varphi$ is defined as $T_{\varphi}f (z)=P_{H^2}\varphi(z)f(z)$, where $P_{H^2}$ denotes the orthogonal projection from $L^2$, the Hilbert space of square integrable measurable functions on $\mathbb T$, onto $H^2$. Toeplitz operators often serve as the elementary building blocks of the study of much complicated operators. Ever since its invention, the Toeplitz operators have been studied in many other Hilbert function spaces on several domains in $\mathbb C^d$, notably the weighted Bergman spaces \cite{AC,ACM,Jewell,SZ1,Vasi} and the $H^p$ spaces, $1\leq p < \infty$ \cite{BH,Vuko}; also see \cite{BS} for a detailed discussion.

The goal of this paper is to extend the theory of Toeplitz operators in the Hardy space setting to another domain in $\mathbb C^d$, $d\geq2$, viz., the symmetrized polydisk. This polynomially convex but not convex domain has attracted a considerable amount of attention for its complex geometry \cite{GS, EZ, NPTZ, NPZ, Niko}, function theory \cite{MSRZ, NTT, TTZ}, and operator theory \cite{SS}. To describe the domain, we consider the {\em{elementary symmetric functions}} $s_i:\mathbb C^d\to\mathbb C$ of degree $i$ in $d$ variables defined as
\begin{align}\label{symm functions}
s_0:=1 \text{ and }s_i(\bm z):=\sum_{1\leq k_1<k_2<\cdots<k_i\leq d}z_{k_1}z_{k_2}\cdots z_{k_i} \text{ for all $\bm z \in \mathbb C^d$ and $1\leq i\leq d$}
\end{align}
and the {\em{symmetrization map}} $\bm s:\mathbb C^d\to\mathbb C^d$ is defined by
\begin{eqnarray}\label{analogue-pi}
\bm s(\bm z):=(s_1(\bm z),s_2(\bm z),\dots,s_d(\bm z))\text{ for all $\bm z \in \mathbb C^d$}.
\end{eqnarray}
The {\em{symmetrized polydisk}}, denoted by $\mathbb G_d$, is defined to be the image of the polydisk $\mathbb D^d$ under the symmetrization map, i.e.,  $\mathbb G_d:=\bm s(\mathbb D^d)$. A typical element of $\mathbb G_d$ will be denoted by $(s_1,s_2,\dots,s_{d-1},p)$. Following Ogle \cite{DJO}, the closure of $\mathbb G_d$ will be denoted by $\Gamma_d:=\bm s(\overline{\mathbb D}^d)$. The symmetrized polydisk came as a natural multivariable generalization of the symmetrized bidisk $\mathbb G_2=\{(z_1+z_2,z_1z_2):(z_1,z_2)\in\mathbb D^2\}$, which, with original motivation coming from the problem of $\mu$-synthesis in robust control (\cite{DP}), has proved to possess rich holomorphic function theory, complex geometry and operator theory; see \cite{ALY-MAMS19, AY_JGA04, AY-JOT03, BS-JFA18, TRY, CYU} and references therein. The algebraic properties of Toeplitz operators on the symmetrized bidisk are studied in \cite{B-D-S}.

As has always been the case in the history of extension of classical theory to multivariable setting, adopting the Toeplitz operators to the symmetrized polydisk has challenges. The main challenge is the absence of a matrix structure of a Toeplitz operator unlike the classical case. We now lay out the main results of this paper in the order they are proved.
\begin{enumerate}
\item In \S \ref{S:HardyG}, we introduce the Hardy space of the symmetrized polydisk $H^2(\mathbb G_d)$ and exhibit two isomorphic copies of it. These isomorphic copies of the Hardy space are used to decipher algebraic properties of the Toeplitz operators.
\item Note that the Toeplitz operators on $H^2$ are precisely those that satisfy $T_z^*TT_z=T$. This is due to Brown and Halmos \cite{BH}. We show that {\em a bounded operator on $H^2(\mathbb G_d)$ is a Toeplitz operator if and only if
    \begin{align}\label{BHnew}
    T_{s_i}^*TT_p=TT_{s_{d-i}} \text{ for each } 1\leq i \leq d-1 \text{ and }T_p^*TT_p=T,
    \end{align}
where $(T_{s_1},T_{s_2},\dots,T_{s_{d-1}},T_p)$ on $H^2(\mathbb G_d)$ is the tuple of the multiplication operator by the coordinate functions.} See Theorem \ref{T:BH} for a proof of the Brown--Halmos type characterization. While the part that proves necessity is somewhat obvious, the proof of the sufficiency part needs substantial analysis. We shall explain in \S \ref{S:Gen} that the role of the tuple $(T_{s_1},T_{s_2},\dots,T_{s_{d-1}},T_p)$ is a direct analogue of that of $T_z$ on $H^2$.

\item Perhaps the simplest class of Toeplitz operators are the analytic Toeplitz operators. The analytic Toeplitz operators are those that come from an analytic symbol $\varphi$, i.e., $\varphi$ has no negative Fourier coefficient in $L^2$.  Several characterizations of analytic Toeplitz operators on $H^2(\mathbb G_d)$ are given in Theorem \ref{T:AnaToep}.

\item Compact operators on $H^2(\mathbb G_d)$ are characterized in terms of certain natural shifts. We then use this to study compact perturbation of Toeplitz operators -- the so called asymptotic Toeplitz operators. This is the content of \S \ref{S:asympT}

\item In \S \ref{S:Gen}, we consider an analogue of a single isometric operator in the symmetrized polydisk case, called a $\Gamma_d$-isometry. These are essentially those tuples $\underline{S}=(S_1,S_2,\dots,S_{d-1},P)$ of commuting bounded Hilbert space operators that have an extension to what is called a $\Gamma_d$-unitary, that is a tuple $\underline{R}=(R_1,R_2,\dots,R_{d-1},U)$ of commuting normal operators with the Taylor joint spectrum lying in the distinguished boundary of $\Gamma_d$. It turns out that $(T_{s_1},T_{s_2},\dots,T_{s_{d-1}},T_p)$ is a prototype example of a $\Gamma_d$-isometry. This motivates us to study those bounded operators that satisfy the Brown--Halmos type relations \eqref{BHnew} with respect to a general $\Gamma_d$-isometry $\underline{S}$ instead of the particular one $(T_{s_1},T_{s_2},\dots,T_{s_{d-1}},T_p)$. We call these the $\underline{S}$-Toeplitz operators or the generalized Toeplitz operators.

   For a symbol $\varphi \in L^\infty$, let $M_\varphi$ denote the multiplication operator on $L^2$. Apart from the concrete algebraic characterization by Brown and Halmos, another way of viewing the Toeplitz operators on $H^2$ is that a Toeplitz operator is precisely the compression of a commutant of $M_{\zeta}$ on $L^2$ with the same norm. We observe a similar phenomenon in the case of an $\underline{S}$-Toeplitz operator. Indeed, given a $\Gamma_d$-isometry $\underline{S}$ on $\cH$, we construct a special minimal $\Gamma_d$-unitary extension $\underline{R}$ on $\cK$ of $\underline{S}$ such that $X$ is an $\underline{S}$-Toeplitz operator if and only if there is a $Y$ in the commutant of $\underline{R}$ whose compression to $\cH$ is $X$ and $\|Y\|=\|X\|$, see part (2) of Theorem \ref{the-iso-case}.

\item For an $L^\infty$ function $\varphi$, the multiplication operator $M_\varphi$ on $L^2$ has the following $2\times2$ matrix form with respect to the decomposition $L^2=H^2\oplus {H^2}^\perp$:
    \begin{align}
    \begin{bmatrix}
                   T_\varphi & H_{\overline{\varphi}}^* \\
                   H_\varphi & DT_\varphi \\
      \end{bmatrix},
      \end{align}where $H_\varphi$ is the Hankel operator and $DT_{\varphi}$ is called the dual Toeplitz operator associated to $\varphi$. We characterize the dual Toeplitz operators on $H^2(\mathbb G_d)$ in \S \ref{S:Dual}.
\end{enumerate}

Finally, the authors would like to express their sincere gratitude to Professor Tirthankar Bhattacharyya for introducing them to this flourishing area of operator theory.

\section{The Hardy space}\label{S:HardyG}

\subsection{Basics of the Hardy Space}
 Let $\bm s$ be the symmetrization map as defined in \eqref{analogue-pi} and $J_{\bm s}$ denote its complex Jacobian.  For $\bm \theta = (\theta_1, \theta_2, \ldots , \theta_d)$ with $\theta_i \in [0,2\pi)$ and $0<r\le1$, let  $r e^{i\bm \theta}$ denote the $d$-tori element $(\,re^{i\theta_1}, \,re^{i\theta_2},\dots,\,re^{i\theta_d})$, while $d\bm\theta$ denotes the normalized product measure on $\mathbb T^d$.
\begin{definition}\label{GnHardy}
 The {\em Hardy space of the symmetrized polydisk}, denoted by  $H^2(\mathbb G_d)$, is the vector space of those holomorphic functions $f$ on $\mathbb G_d$ which satisfy
\[\sup_{\,0<r<1}\int_{\mathbb T^d} |f\circ \bm s (r e^{i\bm \theta})|^2
|J_{\bm s}(r e^{i\bm \theta})|^2 d\bm \theta <\infty .\]

The norm of $f\in H^2(\mathbb G_d)$ is defined to be
$$\|f\|_{H^2(\mathbb G_d)}=\|J_{\bm s}\|^{-1}\Big \{\sup_{0<r<1}\int_{\mathbb
T^d}|f\circ \bm s (r e^{i\bm \theta})|^2
|J_{\bm s}(r e^{i\bm \theta})|^2 d\bm \theta \Big \}^{1/2},$$
where $\|J_{\bm s}\|^2={{\int_{\mathbb T^d}}} |J_{\bm s}(e^{i\bm \theta})|^2 d\bm \theta.$
\end{definition}
We divide by $\|J_{\bm s}\|$ in the definition of the norm in $H^2(\mathbb G_d)$ to make sure that the norm of the constant function $\bm 1$ (sending every element of $\mathbb G_d$ to the constant $1$) is $1$.

The space $H^2(\mathbb G_d)$ as defined above is actually a reproducing kernel Hilbert space and its reproducing kernel was explicitly computed in \cite[Theorem 3.1]{MSRZ}.

The {\em distinguished boundary} of a compact set $K$ in $\mathbb C^d$ is the $\check{\text{S}}$ilov boundary with respect to the algebra $\mathcal A(K)$ of all complex-valued functions continuous on $K$ and holomorphic in the interior of $K$; see Chapter 9 of \cite{AW} for a detailed discussion regarding $\check{\text{S}}$ilov boundary. The distinguished boundary of the symmetrized polydisk, denote by $b\Gamma_d$ is just the symmetrization of the $d$-tori $\mathbb T^d$, i.e., $b\Gamma_d = \bm s(\mathbb T^d).$  See \cite[Theorem 2.4]{SS} for characterizations of $b\Gamma_d$.

Consider the measure $\mu$ defined on a Borel subset $E$ of $\mathbb T^d$ induced by the Jacobian $J_{\bm s}$ of the symmetrization map $\bm s$ as
$$
\mu(E)=\int_{E}|J_{\bm s}(e^{i\bm \theta})|^2 d\bm\theta.
$$Let us denote by $\nu$ the push forward measure on $b\Gamma_d$ of $\mu$ and consider the $L^2$ space of $b\Gamma_d$ with respect to it, i.e.,
$$
L^2(b\Gamma_d):=L^2(b\Gamma_d,\nu)=\{f:b\Gamma_d \to \mathbb C: \int_{\mathbb T^d} |f\circ \bm s (e^{i\bm \theta})|^2
|J_{\bm s}(e^{i\bm \theta})|^2 d\bm\theta<\infty  \}.
$$
Just like in the case of the classical Hardy space, Theorem \ref{Iso1} below shows that $H^2(\mathbb G_d)$ is a subspace of $L^2(b\Gamma_d)$ with respect to a natural identification.

We shall first set up some standard notations which will be used in what follows.
\begin{itemize}

\item A {\em{partition}} $\bm p$ of $\mathbb Z$ of size $d$ is a $d$-tuple $(p_1,p_2,\dots,p_d)$ of integers such that $p_1> p_2>\cdots> p_d$.
Let
 $$\llbracket z \rrbracket:=\{(p_1,p_2,\dots,p_d): p_1> p_2>\cdots> p_d \text{ and } p_j \in \mathbb Z \}$$ and
 $$\llbracket n \rrbracket:=\{(p_1,p_2,\dots,p_d): p_1> p_2>\cdots> p_d \geq 0  \};$$
\item $\Sigma_d$ denotes the permutation group of the set $\{1,2,\dots,d\}$.
\item For $\sigma\in\Sigma_d$, $\bm m=(m_1,m_2,\dots, m_d) \in \llbracket z \rrbracket$ and $\bm z=(z_1,z_2,\dots z_d) \in \mathbb C^d$, let $\bm z_\sigma:=(z_{\sigma(1)},z_{\sigma(2)},\dots z_{\sigma(d)})$, $\bm {z^m}:=z_1^{m_1}z_2^{m_2}\cdots z_d^{m_d}$.
\end{itemize}
Therefore for a partition $\bm m$ in $\llbracket z \rrbracket$, $\sigma$ in $\Sigma_n$ and $\bm z=(z_1,z_2,\dots z_d)$ in $\mathbb C^d$,
 $$
 \bm z^{\bm m_\sigma}=z_1^{m_{\sigma(1)}}z_2^{m_{\sigma(2)}}\cdots z_d^{m_{\sigma(d)}}.
 $$
For  $\bm m\in \llbracket n \rrbracket$, consider the function $\bm{a_m(z)}$ defined by
   $$
   \bm{a_m(z)}:=\sum_{\sigma\in\Sigma_d}\text{sgn}(\sigma)\bm z^{\bm m_\sigma}=\det\left(((z_i^{m_j}))_{i,j=1}^d\right),
   $$where $\text{sgn}(\sigma)$ denotes the sign of the permutation $\sigma$. The above representation of $\bm{a_m}$ shows that it is zero if and only if there exist $1\leq i< j \leq d$ such that $m_i=m_j$. Note that the functions $\bm{a_m}$ continue to be defined on $\mathbb T^d$ for $\bm m\in \llbracket z \rrbracket$ and that $\bm{a_m}(\bm z_\sigma)=\operatorname{sgn}(\sigma)\bm{a_m}(\bm z)$ for every $\sigma$ in $\Sigma_n$. This shows that each $\bm{a_m}$ is a member of the following subspace of $L^2(\mathbb T^d)$ consisting of anti-symmetric functions
$$L^2_{\rm anti} (\mathbb T^d):=\{f\in L^2(\mathbb T^d):f( e^{i\bm \theta_\sigma})=\text{sgn}(\sigma)f(e^{i\bm \theta})\text{ a.e. w.r.t. the product measure}\}.$$ Moreover, the set
$$
 \{\bm{a_m}:\bm m \in \llbracket z \rrbracket\}
 $$ actually forms an orthogonal basis for $L^2_{\rm anti} (\mathbb T^d)$, which follows from the facts that the operator $\mathbb P:L^2(\mathbb T^d)\to L^2_{\text{anti}}(\mathbb T^d)$ defined by
$$
\mathbb P(f)(e^{i\bm \theta})=\frac{1}{d!}\sum_{\sigma\in \Sigma_d} \text{sgn}(\sigma)f(e^{i\bm \theta_\sigma})
$$ is an orthogonal projection onto $L^2_{\text{anti}}(\mathbb T^d)$ and that the set $\{e^{i\bm m\bm \cdot\bm\theta}:\bm m\in \mathbb Z^n\}$ is an orthonormal basis for $L^2(\mathbb T^d)$. Here $\bm m\cdot\bm\theta:=(m_1\theta_1,m_2\theta_2,\dots,m_d\theta_d)$ for $\bm m=(m_1,m_2,\dots,m_d)$ and $\bm\theta=(\theta_1,\theta_2,\dots,\theta_d)$.

Consider the following subspace of the Hardy space $H^2(\mathbb D^d)$ over the $d$-disk
$$H^2_{\rm anti} (\mathbb D^d) := \{ f \in H^2(\mathbb D^d) : f(\bm {z}_\sigma) = \text{sgn}(\sigma) f(\bm z), \text{ for all } \bm z\in \mathbb D^d \text{ and }\sigma\in\Sigma_d \}.$$ It was observed in \cite[Sec.\ 3]{MSRZ} that the set
 $$
 \mathcal B=\{\bm{a_p}:\bm p \in \llbracket n \rrbracket\}
 $$is an orthogonal basis for $H^2_{\rm anti} (\mathbb D^d)$. The following theorem immediately allows us to consider boundary values of the Hardy space functions.

\begin{thm} \label{Iso1}
There is an isometric embedding of the space $H^2(\mathbb G_d)$ inside $L^2(b\Gamma_d)$.
\end{thm}
\begin{proof}
Define $U_h: H^2(\mathbb G_d) \to H^2_{\rm anti} (\mathbb D^d)$ by
\begin{eqnarray}\label{unotilde}
U_h (f) = \frac{1}{\| J_{\bm s} \|} J_{\bm s}(\cdot)f \circ \bm s(\cdot), \text{ for all } f \in H^2(\mathbb G_d)
\end{eqnarray}
and ${U_l}:L^2(b\Gamma_d) \to L^2_{\rm anti} (\mathbb D^d)$ by
\begin{eqnarray}\label{utilde}
 U_l f=\frac{1}{\|J_{\bm s}\|} J_{\bm s}(\cdot)f \circ \bm s(\cdot), \text{ for all } f \in L^2(b\Gamma_d).
\end{eqnarray}
It follows from the definitions of the norms in the respective spaces that both $U_l$ and $U_h $ are isometries. Moreover, since they are easily seen to be surjective, they are unitary. Also note that there is an isometry $W: H^2_{\rm anti} (\mathbb D^d) \to L^2_{\rm anti}(\mathbb T^d)$ which sends a function to its radial limit. Therefore we have the following commutative diagram:
$$
\begin{CD}
H^2(\mathbb G_d) @> U_l^{-1}\circ W \circ U_h >>  L^2(b\Gamma_d)\\
@VU_h  VV @VV U_l V\\
H^2_{\rm anti} (\mathbb D^d)@>>W> L^2_{\rm anti}(\mathbb T^d)
\end{CD}.
$$Consequently, the embedding of $H^2(\mathbb G_d)$ into $L^2(b\Gamma_d)$ is carried out by the isometry $U_l^{-1}\circ W \circ U_h$.
\end{proof}
Let $L^\infty(b\Gamma_d)$ denote the algebra of all functions on $b\Gamma_d$ which are bounded almost everywhere on $b\Gamma_d$ with respect to the measure $\nu$.
For a function $\varphi$ in $L^\infty(b\Gamma_d)$, let $M_\varphi$ be the operator on $L^2(b \Gamma_d)$ defined by
$$
M_\varphi f (s_1,\dots,s_{d-1},p) = \varphi(s_1,\dots,s_{d-1},p)f(s_1,\dots,s_{d-1},p),
$$ for all $f$ in $L^2(b \Gamma_d)$.
We note that the tuple $(M_{s_1},\dots,M_{s_{d-1}},M_p)$ of multiplication by coordinate functions on $L^2(b \Gamma_d)$ is a commuting tuple of normal operators.
\begin{definition}
For a function $\varphi$ in $L^\infty(b\Gamma_d)$, the multiplication operator $M_\varphi$ is called the {\em Laurent operator} with symbol $\varphi$. The compression of $M_\varphi$ to $H^2(\mathbb G_d)$ is called the {\em Toeplitz operator} and denoted by $T_\varphi$. Therefore
$$
T_\varphi f = Pr M_\varphi f \;\text{  for all $f$ in $H^2( \mathbb G_d)$,}
$$where $Pr$ denotes the orthogonal projection of $L^2(b\Gamma_d)$ onto $H^2(\mathbb G_d)$.
\end{definition}

Let $L^\infty_{\text{sym}}(\mathbb T^d)$ denote the sub-algebra of $L^\infty(\mathbb T^d)$ consisting of symmetric functions, i,e., $f(e^{i\bm{\theta}_\sigma})=f(e^{i\bm{\theta}})$ a.e., and $\Pi_{\bm s}:L^\infty(b\Gamma_d)\to L^\infty_{\text{sym}}(\mathbb T^d)$ be the $*$-isomorphism defined by
$$
\varphi \mapsto \varphi \circ \bm s
$$
where $\bm s$ is the symmetrization map as defined in (\ref{analogue-pi}). Induced by the unitary ${U_l}$ as in \eqref{utilde}, let $\Pi_l:\mathcal B(L^2(b\Gamma_d))\to \mathcal B(L^2_{\text{anti}}(\mathbb T^d))$ be the map $$\Pi_l:X\mapsto {U_l} X {U_l}^*.$$
\begin{thm}\label{equiv-Laurent}
Let $\Pi_1$ and $\Pi_2$ be the above $*$-isomorphisms. Then the following diagram is commutative:
$$
\begin{CD}
L^\infty(b\Gamma_d) @> \Pi_{\bm s}>>  L^\infty_{\operatorname{sym}}(\mathbb T^d)\\
@Vi_1 VV @VVi_2 V\\
\mathcal B(L^2(b\Gamma_d))@>>\Pi_l> B(L^2_{\operatorname{anti}}(\mathbb T^d))
\end{CD},
$$where $i_1$ and $i_2$ are the canonical inclusion maps. Equivalently, for $\varphi \in L^\infty(b\Gamma_d)$, the operators $M_\varphi$ on $L^2(b\Gamma_d)$ and $M_{\varphi\circ \bm s}$ on $L^2_{\operatorname{anti}}(\mathbb T^d)$ are unitarily equivalent via the unitary $ U_l$ as in \eqref{utilde}.
\end{thm}
\begin{proof}
To show that the above diagram commutes all we need to show is that $U_l M_\varphi  U_l^*=M_{\varphi \circ \bm s}$, for every $\varphi$ in $L^\infty(b\Gamma_d)$. This follows from the following computation: for every $\varphi$ in $L^\infty(b\Gamma_d)$ and $f \in L^2_{\text{anti}}(\mathbb T^d)$,
$$
 U_l M_\varphi  U_l^*(f)= U_l (\varphi U_l^* f)=( \varphi\circ \bm s) \frac{1}{\|J_{\bm s}\|}J_{\bm s}( U_l^*f\circ \bm s )= M_{\varphi\circ \bm s}(f).
$$
\end{proof}
As a consequence of the above theorem, we obtain that the Toeplitz operators on the Hardy space of the symmetrized polydisk are unitarily equivalent to those on $H^2_{\text{anti}}(\mathbb D^d)$.
\begin{corollary}\label{equiv-Toep}
For a $\varphi \in L^\infty(b\Gamma_d)$, $T_\varphi$ is unitarily equivalent to $T_{\varphi\circ \bm s}:=P_aM_{\varphi \circ \bm s}|_{H^2_{\operatorname{anti}}(\mathbb D^d)}$, where $P_a$ stands for the projection of $L^2_{\operatorname{anti}}(\mathbb T^d)$ onto $H^2_{\operatorname{anti}}(\mathbb D^d)$.
\end{corollary}
\begin{proof}
This follows from the fact that the operators $M_\varphi$ and $M_{\varphi \circ \bm s}$ are unitarily equivalent via the unitary $U_h$ as in \eqref{unotilde}. Indeed for every $f\in H^2(\mathbb G_d)$
$$
U_hM_{\varphi}f=\frac{1}{\|J_{\bm s}\|}\varphi\circ\bm s(\cdot)f\circ\bm s(\cdot)J_{\bm s}(\cdot)=M_{\varphi\circ\bm s}U_hf.
$$
\end{proof}
\begin{notation}
It is trivial to see that the coordinate functions $s_1,\dots,s_{d-1},p$ are in $L^\infty(b\Gamma_d)$. Just like in the classical case, the Toeplitz operators $T_{s_1},T_{s_2},\dots, T_{s_{d-1}},T_p$ will play a significant role in describing a general Toeplitz operator on $H^2(\mathbb G_d)$. {\em For each $j=1,2,\dots, d-1$, we shall denote by $T_{s_j(\bm z)}$ and $T_{s_d(\bm z)}$ on $H^2_{\operatorname{anti}}(\mathbb D^d)$ the unitary copies of $T_{s_j}$ and $T_p$, respectively. Similarly, we use the notations $M_{s_1(\bm z)},M_{s_2(\bm z)},\dots,M_{s_d(\bm z)}$ to denote the unitary copies of $M_{s_1},M_{s_2},\dots,M_{p}$, respectively.}
\end{notation}

We end the subsection with one more isomorphic copy of the Hardy space $H^2(\mathbb G_d)$. If $\mathcal E$ is a Hilbert space, let $\mathcal O (\mathbb D , \mathcal E)$
be the class of all $\mathcal E$-valued holomorphic functions on $\mathbb D$. Let
$$H^2_{\mathcal E}(\mathbb D) = \{ f(z) = \sum a_k z^k \in \mathcal O (\mathbb D , \mathcal E) : a_k \in \mathcal E \mbox{ with } \| f \|^2 = \sum \| a_k \|^2 < \infty \}.$$
\begin{lemma}\label{vector-valued-Hardy} There is a Hilbert space isomorphism $U_1$ from $H^2_{\operatorname{anti}}(\mathbb D^d)$ onto the vector-valued Hardy space $H^2_{\mathcal E}(\mathbb D)$, where
$$\mathcal E = \overline{\operatorname{span}}\{ \bm{a_p(z)}: \bm p\in \llbracket n \rrbracket \text{ such that } \bm p=(p_1,p_2,\dots,p_{d-1},0)\} \subset H^2_{\operatorname{anti}}(\mathbb D^d).$$
    Moreover, this unitary $U_1$ intertwines $T_{s_d(\bm z)}$ on $H^2_{\operatorname{anti}}(\mathbb D^d)$ with the unilateral shift of infinite multiplicity $T_z$ on $H^2_{\mathcal E}(\mathbb D)$. \end{lemma}

\begin{proof}
As observed above, the set $\{ \bm{a_p(z)}: \bm p\in \llbracket n \rrbracket \text{ such that } \bm p=(p_1,p_2,\dots,p_{d-1},0) \}$ forms an orthogonal basis in $\mathcal E$. Thus
$$
\{ z^q \bm a_{\bm p} : q\geq 0 \text{ and }\bm p\in \llbracket n \rrbracket \text{ such that } \bm p=(p_1,p_2,\dots,p_{d-1},0) \}
$$
is an orthogonal basis for $H^2_{\mathcal E}(\mathbb D)$. On the other hand, the space $H^2_{\text{anti}}(\mathbb D^d)$ is spanned by the orthogonal set
$\{\bm{a_p(z)}: \bm p\in \llbracket n \rrbracket\}$ and if $\bm p=(p_1,p_2,\dots,p_{d-1},p_d)$, then
$$\bm {a_p(z)}=(z_1z_2\cdots z_d)^{p_d}\bm{a_{\tilde p}(z)}={T^{p_d}_{s_d(\bm z)}}\bm{a_{\tilde p}(z)},$$
where $\bm{\tilde p}=(p_1-p_d,p_2-p_d,\dots,p_{d-1}-p_d,0)$. Define the unitary operator from $H^2_{\text{anti}}(\mathbb D^d)$ onto $H^2_{\mathcal E}(\mathbb D)$ by
$$
U_1:\bm {a_p(z)} \mapsto z^{p_d}\bm{a_{\tilde p}(z)},
$$and then extending linearly. This preserves norms because $T_{s_d(\bm z)}$ is an isometry on $H^2_{\text{anti}}(\mathbb D^d)$ and $T_z$ is an isometry on $H^2_{\mathcal E}(\mathbb D)$. It is surjective and obviously intertwines $T_{s_d(\bm z)}$ and $T_z$.
\end{proof}
By virtue of the isomorphisms $U_h$ and $U_1$ described above, we have the following commutative diagram:
$$\begin{tikzcd}
(H^2(\mathbb G_d) , T_p) \arrow{r}{U_h  }  \arrow{rd}{U_2}
  & (H^2_{\text{anti}}(\mathbb D^d) , T_{s_d(\bm z)}) \arrow{d}{U_1} \\
    & (H^2_{\mathcal E}(\mathbb D) , T_z)
\end{tikzcd}$$
i.e., the operator $T_p$ on $H^2(\mathbb G_d)$ is unitarily equivalent to the unilateral shift $T_z$ on the vector valued Hardy space
$H^2_{\mathcal E}(\mathbb D)$ via the unitary $U_2$. We call $\mathcal E$ the {\em{coefficient space}} of the symmetrized polydisk. It is a subspace of $H^2_{\mathcal E}(\mathbb D)$ which
is naturally identifiable with the subspace of constant functions in $H^2_{\mathcal E}(\mathbb D)$.

\section{An algebraic characterization of Toeplitz operators and their basic properties}

%

Although a Toeplitz operator is defined in terms of an $L^\infty$ function, it is a difficult question of how to recognize a given bounded operator $T$ on the relevant Hilbert space as a Toeplitz operator. This question was answered for the Hardy space of the unit disc by Brown and Halmos in Theorem 6 of \cite{BH} where they showed that $T$ has to be invariant under conjugation by the unilateral shift. We show that in our context one needs not only $T_p$ but all of $T_{s_i}$, $i=1,2,\dots,d-1$  to give such a characterization.

We first note the following preliminary results each of which will play a significant role in obtaining a Brown--Halmos type relations for $H^2(\mathbb G_d)$.
\begin{lemma}
Let $M$ be a bounded operator on $L^2(b\Gamma_d)$ that commutes with each entry of the tuple $(M_{s_1}, M_{s_2},\dots,M_p)$. Then $M$ is of the form $M_\varphi$, for some function $\varphi$ in $L^\infty(b\Gamma_d)$.
\end{lemma}
\begin{proof}
Since $(M_{s_1}, M_{s_2},\dots,M_p)$ is a commuting tuple of normal operators, the Taylor joint spectrum $\sigma_T(M_{s_1}, M_{s_2},\dots,M_p)=b\Gamma_d$. Now it follows from the spectral theorem for commuting normal operators that the von Neumann algebra generated by $\{M_{s_1}, M_{s_2},\dots,M_p\}$ is $L^\infty(b\Gamma_d)$, which is a maximal abelian von Neumann algebra, i.e., $\{M_{s_1}, M_{s_2},\dots,M_p\}'$ is $*$-isomorphic to $L^\infty(b\Gamma_d)$.
\end{proof}
In view of identification results Theorem \ref{Iso1} and Theorem \ref{equiv-Laurent}, the following is then an obvious consequence of the above lemma.
\begin{corollary}\label{polydisk-lauraoperator}
Any bounded operator $M$ on $L^2_{\operatorname{anti}}(\mathbb T^d)$ that commutes with each entry of the tuple $(M_{s_1(\bm z)}, M_{s_2(\bm z)},\dots,M_{s_d(\bm z)})$ is of the form $M_\varphi$, for some function $\varphi$ in $L^\infty_{\operatorname{sym}}(\mathbb T^d)$.
\end{corollary}
The following algebraic relations of the multipliers $M_{s_1},M_{s_2},\dots,M_{p}$ on $L^2(b\Gamma_d)$ and their restrictions to $H^2(\mathbb G_d)$ give us a major lead to the Brown--Halmos relations.
\begin{lemma}\label{L:GammaUni}
For every $j=1,2,\dots d-1$, $M_{s_j}^*M_p=M_{s_{d-j}}$ and $T_{s_j}^*T_p=T_{s_{d-j}}$.
\end{lemma}
\begin{proof}
From the definition of the elementary symmetric functions $s_j$, it follows that if $\bm z\in\mathbb T^d$, then one has
$$s_{d-j}(\bm z)=\overline{s_j(\bm z)}s_d(\bm z)$$and hence follows the first set of equations. For the second part we choose $f,g$ in $H^2(\mathbb G_d)$ and compute using the first part
$$
\langle T_{s_j}^*T_p f,g  \rangle = \langle T_p f,T_{s_j}g  \rangle=\langle M_{s_j}^*M_p f,g  \rangle=\langle M_{s_{d-j}}f,g  \rangle=\langle T_{s_{d-j}}f,g  \rangle.
$$ This proves the lemma.
\end{proof}
\begin{thm}\label{T:BH}
Let $T$ be a bounded operator on $H^2(\mathbb G_d)$. Then $T$ is a Toeplitz operator if and only if for each $i=1,2,\dots,d-1$,
\begin{eqnarray}\label{Toeplitzcharc}
 T_{s_i}^*TT_p=TT_{s_{d-i}} \text{ and } T_p^*TT_p=T.
\end{eqnarray}
\end{thm}
\begin{proof}
Let us first prove the easier direction -- the `only if' part. Let $T$ be a Toeplitz operator with symbol $\varphi$. Then for $f,g\in H^2(\mathbb G_d)$,
\begin{eqnarray*}
\langle T_p^*T_\varphi T_p f,g \rangle = \langle T_\varphi T_p f,T_p g \rangle = \langle Pr M_\varphi T_p f, T_p g \rangle &=& \langle M_\varphi M_p f, M_p g \rangle =\langle M_\varphi  f,  g \rangle \\ &=& \langle Pr M_\varphi  f,  g \rangle = \langle T_\varphi  f,  g \rangle.
\end{eqnarray*}
Also for $i=1,2,\dots,d-1$, using Lemma \ref{L:GammaUni} we compute
\begin{eqnarray*}
\langle T_{s_i}^*T_\varphi T_p  f,  g \rangle_{H^2} &=& \langle Pr M_\varphi T_p  f,  T_{s_i} g \rangle_{H^2} = \langle M_\varphi M_p  f,  M_{s_i} g \rangle_{L^2} = \langle M_{s_i}^*M_p M_\varphi f, g \rangle_{L^2}
\\&=& \langle M_\varphi M_{s_{d-i}} f, g \rangle_{L^2} =\langle Pr M_\varphi M_{s_{d-i}} f, g \rangle_{H^2} =\langle T_\varphi T_{s_{d-i}} f, g \rangle_{H^2}.
\end{eqnarray*}

For the sufficient part, we work on $H^2_{\rm anti}(\mathbb D^d)$ and invoke Corollary \ref{equiv-Toep} to draw the conclusion. So, let $T$ be a bounded operator on $H^2_{\rm anti}(\mathbb D^d)$ that satisfies $ T_{s_i(\bm z)}^*TT_{s_d(\bm z)}=TT_{s_{d-i}(\bm z)}$ and $T_{s_d(\bm z)}^*TT_{s_d(\bm z)}=T$, where note that $s_i$ is the $i$-th symmetric function defined in (\ref{symm functions}). For a partition $\bm p=(p_1,p_2,\dots,p_d)$ in $\llbracket z \rrbracket$, note that for $r \geq 0$, $M_{s_d(\bm z)}^r\bm{a_p}=\bm{a_{p+r}}$, where $\bm r$ is the $d$-tuple $(r,r,\dots,r)$ and $\bm p+ \bm r =(p_1+r,p_2+r,\dots,p_d+r)$. {\em Therefore for every partition $\bm p$ in $\llbracket z \rrbracket$, there exists a sufficiently large $r$ (depending on $\bm p$) such that $M_{s_d(\bm z)}^n\bm{a_p} \in H^2_{\rm anti}(\mathbb D^d)$.} For each non-negative integer $r$, define an operator $T_r$ on $L^2_{\text{anti}}(\mathbb T^d)$ by
$$
T_r:={M^{* r}_{s_d(\bm z)}}TP_aM_{s_d(\bm z)}^r,
$$where $P_a$ is the orthogonal projection of $L^2_{\text{anti}}(\mathbb T^d)$ onto $H^2_{\text{anti}}(\mathbb D^d)$.
Let $\bm p$ and $\bm q$ be two partitions in $\llbracket z \rrbracket$, then for sufficiently large $r$, we have
\begin{eqnarray}\label{auxtoepcharcone}
\langle T_r\bm{a_p}, \bm{a_q} \rangle = \langle T M_{s_d(\bm z)}^r\bm{a_p}, M_{s_d(\bm z)}^r\bm{a_q} \rangle = \langle T \bm{a_{p+r}}, \bm{a_{q+r}} \rangle.
\end{eqnarray}
Since $T_{s_d(\bm z)}^*TT_{s_d(\bm z)}=T$, we have for every $r\geq 0$, $T_{s_d(\bm z)}^{*r}TT_{s_d(\bm z)}^r=T$. Let $\bm p$ and $\bm q$ be two partitions in $\llbracket n \rrbracket$, then for every $r\geq 0$,
\begin{eqnarray}\label{auxtoepcharctwo}
\langle T\bm{a_p}, \bm{a_q} \rangle = \langle T T_{s_d(\bm z)}^r\bm{a_p}, T_{s_d(\bm z)}^r \bm{a_q} \rangle=\langle T \bm{a_{p+r}}, \bm{a_{q+r}} \rangle.
\end{eqnarray}
Since $\{\bm{a_p}: \bm p\in \llbracket z \rrbracket\}$ is an orthogonal basis for $L^2_{\rm anti}(\mathbb T^d)$ and the sequence of operators $\{T_r\}$ on $L^2_{\rm anti}(\mathbb T^d)$ is uniformly bounded by $\|T\|$, by (\ref{auxtoepcharcone}) and $(\ref{auxtoepcharctwo})$ the sequence $\{T_r\}$ converges weakly to some operator $T_\infty$ (say) acting on $L^2_{\rm anti}(\mathbb T^d)$.

We now use Corollary \ref{polydisk-lauraoperator} to conclude that $T_\infty=M_\varphi$, for some $\varphi \in L^\infty_{\text{sym}}(\mathbb T^d)$. Therefore we have to show that $T_\infty$ commutes with each entry of $(M_{s_1(\bm z)},M_{s_2(\bm z)},\dots,M_{s_d(\bm z)})$. That $T_\infty$ commutes with $M_{s_d(\bm z)}$ is clear from the definition of $T_\infty$. The following computation shows that $T_\infty$ commutes with the other entries also. Let $\bm p$ and $\bm q$ be two partitions in $\llbracket z \rrbracket$ and $1\leq i \leq d-1$. Then
\begin{eqnarray*}
\langle M_{s_i(\bm z)}^*T_\infty^* \bm{a_p}, \bm{a_q} \rangle
&=& \lim_{r}\langle M_{s_i(\bm z)}^*M_{s_d(\bm z)}^{* r} T^*P_aM_{s_d(\bm z)}^r\bm{a_p}, \bm{a_q}\rangle
\\
&=& \lim_{r}\langle T_{s_i(\bm z)}^*T^*M_{s_d(\bm z)}^r\bm{a_p},M_{s_d(\bm z)}^r\bm{a_q} \rangle \;\;(\text{for sufficiently large $r$})
\\
&=& \lim_{r}\langle T_{s_d(\bm z)}^*T^*T_{s_{d-i}(\bm z)} M_{s_d(\bm z)}^r\bm{a_p},M_{s_d(\bm z)}^r\bm{a_q} \rangle \;\;(\text{applying (\ref{Toeplitzcharc})})
\\
&=& \lim_{r}\langle M_{s_d(\bm z)}^{* r+1}T^*P_aM_{s_d(\bm z)}^{r+1}M^*_{s_d(\bm z)}M_{s_{d-i}(\bm z)} \bm{a_p},\bm{a_q} \rangle
\\
&=& \lim_{r}\langle M_{s_d(\bm z)}^{* r+1}P_aT^*P_aM_{s_d(\bm z)}^{r+1}M_{s_i(\bm z)}^* \bm{a_p},\bm{a_q} \rangle \;\;(\text{by Lemma \ref{L:GammaUni}})
\\
&=& \langle T_\infty^*M_{s_i(\bm z)}^* \bm{a_p}, \bm{a_q} \rangle.
\end{eqnarray*}
Therefore there exists a $\varphi\in L^\infty_{\text{sym}}(\mathbb T^d)$ such that $T_\infty=M_\varphi$. Now for $f$ and $g$ in $H^2_{\text{anti}}(\mathbb D^d)$, we have
\begin{eqnarray*}
\langle P_aM_\varphi f,g \rangle = \langle M_\varphi f,g \rangle = \langle T_\infty f,g \rangle &=& \lim_{r}\langle T_r f,g \rangle
=\lim_{r} \langle T  T_{s_d(\bm z)}^r f,T_{s_d(\bm z)}^rg \rangle =  \langle Tf,g \rangle.
\end{eqnarray*}
Hence $T$ is the Toeplitz operator with symbol $\varphi$.
\end{proof}
The following is a straightforward consequence of Lemma \ref{L:GammaUni} and the characterization of Toeplitz operators obtained above.
\begin{corollary} \label{CommutesHenceToeplitz}
If $T$ is a bounded operator on $H^2(\mathbb G_d)$ that commutes with each entry of $(T_{s_1},\dots,T_{s_{d-1}}, T_p)$, then $T$ is a Toeplitz operator.
\end{corollary}
\begin{definition}
The relations in \eqref{Toeplitzcharc} will be referred to as  the {\em Brown--Halmos relations}.
\end{definition}
We progress with basic properties of Toeplitz operators. It is a natural question whether any of the Brown--Halmos relations implies the other. We give here an example of an operator $Y$ which satisfies none other than the last one.
\begin{example} \label{FO}
We actually construct an example on the space $H^2_{\rm anti} (\mathbb D^d)$ and invoke the unitary equivalences we established in Section 3 to draw the conclusion. For each $1\leq j \leq d-1$, we define operators $Y_j$ on $H^2_{\text{anti}} (\mathbb D^d)$ by defining its action on the basis elements $\{\bm{a_p}:\bm p \in \llbracket n \rrbracket\}$ as
\begin{equation}\label{the-other-shift}
Y_j \bm{a_p} = \bm{a_{p+f_j}}, \text{ where } \bm{f_j}=(\underbrace{1,\dots,1}_{j\text{-times}},0\dots,0).
 \end{equation}
Note that $M_{s_d(\bm z)}\bm{a_p}=\bm{a_{p+f_d}}$ for every $\bm p\in \llbracket n \rrbracket$. Therefore each $Y_j$ commutes with $M_{s_d(\bm z)}|_{H^2_{\rm anti} (\mathbb D^d)}$. Thus by the algebraic relations in Lemma \ref{L:GammaUni}, it follows that for $Y_j$ to satisfy other Brown--Halmos relations, it is necessary for it to commute with each $T_{s_i(\bm z)}$. But we show below that $Y_j$ commutes with none of the $T_{s_i(\bm z)}$.

To that end let us fix some $j$ between $1$ and $d-1$. Choose the following partition
$$
\bm p=(p_1,\dots,p_{j-1},p_{j+1}+1,p_{j+1},\dots,p_d)
$$
in a way so that the difference between each entry with its next entry is strictly greater than one (except the $j$-th entry). For this $\bm p$, note that the difference between each adjacent entries of $\bm p+\bm f_j$ is greater than $2$. Hence we conclude that $T_{s_i(\bm z)}Y_j\bm{a_p}$ has exactly $n!$ many non-zero terms. On the other hand, note that the $i$-the elementary symmetrization function corresponds to the partition
$$\bm m^i=(\underbrace{1,\dots,1}_{i-\text{times}},0,\dots,0)$$
in the sense that $s_i=\bm s_{\bm m^i}(\bm z)$. Therefore $T_{s_i(\bm z)}\bm{a_p}=s_i\bm{a_p}=\bm s_{\bm m^i}\bm{a_p}=\sum_{\sigma\in \Sigma_d}\bm{a_{p+\bm m^i_\sigma}}$. Note that since $i<n$, we can choose a $\sigma$ form $\Sigma_d$ so that $\bm m^i_{\sigma}$ has $0$ and $1$ in its $j$-th and $(j+1)$-th entry, respectively. For this $\sigma$, it follows that $\bm{a}_{\bm p+\bm m^i_\sigma}=0$. Hence $T_{s_i(\bm z)}\bm{a_p}$ has at most $n!-1$ many non-zero terms and hence so does $Y_jT_{s_i(\bm z)}\bm{a_p}$. Therefore $Y_j$ and $T_{s_i(\bm z)}$ can not commute.
\end{example}

\begin{lemma}\label{L:Injection}
For $\varphi\in L^{\infty}(b\Gamma_d)$ if $T_\varphi$ is the zero operator, then $\varphi=0$, a.e. In other words, the map $\varphi \mapsto T_\varphi$ from $L^\infty (b\Gamma_d)$ into the set of all Toeplitz operators on the symmetrized polydisk, is injective.
\end{lemma}
\begin{proof}
Let $\varphi\circ \bm s(\bm z)=\sum_{\bm m\in\mathbb Z^d}\alpha_{\bm m}\bm z^{\bm m}\in L^\infty_{\text{sym}}(\mathbb T^d)$. Since $\varphi\circ\bm s$ is symmetric, the coefficients $\alpha_{\bm m}$ are such that $\alpha_{\bm m_\sigma}=\alpha_{\bm m}$ for every $\sigma\in\Sigma_d$. Also, $\varphi\circ\bm s$ in $L^2(\mathbb T^d)$ implies the coefficient sequence $\{\alpha_{\bm m}\}$ is square-summable. Suppose that $T_{\varphi\circ \bm s}$ on $H^2_{\text{anti}}(\mathbb D^d)$ is the zero operator. Then for every $\bm p$ and $\bm q$ in $\llbracket n\rrbracket $  we have
\begin{eqnarray}
0=\langle T_{\varphi\circ \bm s} \bm {a_p}, \bm{a_q}\rangle&=&\langle \sum_{\bm m\in\mathbb Z^d}\alpha_{\bm m}\bm {z^{m}}\sum_{\sigma\in\Sigma_d}\operatorname{sgn}(\sigma)\bm z^{\bm p_\sigma},\sum_{\delta\in\Sigma_d}\operatorname{sgn}(\delta)\bm z^{\bm q_\delta} \rangle\notag\\
&=&\langle \sum_{\bm m\in\mathbb Z^d,\sigma\in\Sigma_d}\operatorname{sgn}(\sigma)\alpha_{\bm m}\bm z^{\bm m+\bm p_\sigma}, \sum_{\delta\in\Sigma_d}\operatorname{sgn}(\delta)\bm z^{\bm q_\delta}\rangle \notag\\
&=& \sum_{\bm m\in\mathbb Z^d,\sigma,\delta\in\Sigma_d}\operatorname{sgn}(\sigma)\operatorname{sgn}(\delta)\alpha_{\bm m}\bm \langle z^{\bm m+\bm p_\sigma},\bm z^{\bm q_\delta} \rangle \notag\\
&=& \sum_{\sigma,\delta\in\Sigma_d}\operatorname{sgn}(\sigma)\operatorname{sgn}(\delta)\alpha_{\bm q_\delta-\bm p_\sigma}\notag\\
&=& \sum_{\sigma,\delta\in\Sigma_d}\operatorname{sgn}(\sigma)\operatorname{sgn}(\delta)\alpha_{(\bm q_{\sigma^{-1}\delta}-\bm p)_\sigma}\notag\\
&=& \sum_{\sigma,\delta\in\Sigma_d}\operatorname{sgn}(\sigma)\operatorname{sgn}(\delta)\alpha_{\bm q_{\sigma^{-1}\delta}-\bm p}
= d!\sum_{\sigma\in\Sigma_d}\operatorname{sgn}(\sigma)\alpha_{\bm q_{\sigma}-\bm p}\label{Baapre!}.
\end{eqnarray}
Our claim is that the above sum being zero for every $\bm p$ and $\bm q$ in $\llbracket n\rrbracket$ implies $\alpha_{\bm m}=0$ for every $\bm m\in\mathbb Z^d$. Since $\alpha_{\bm m}$ remains the same for different permutation of $\bm m$, without loss of generality, we choose an $\bm m=(m_1,m_2,\dots,m_d)$ such that
$$m_1\leq m_2\leq \cdots \leq m_d\leq r \;(\text{say}).$$ By choice of $r$, one can assume it to be bigger or equal to $-1$. Choose $\bm q=(r+d,\dots,r+1)$ and $\bm p= \bm q - \bm m$. Then both $\bm p$ and $\bm q$ are in $\llbracket n\rrbracket$. With this choice we have from the above inner product computation that
$$
\alpha_{\bm m}= - \sum_{e\neq\sigma\in\Sigma_d}\operatorname{sgn}(\sigma)\alpha_{\bm q_{\sigma}-\bm p},
$$
where $e$ is the identity permutation. Fix a positive integer $n$ and consider the following element in $\llbracket n\rrbracket$
$$
\bm n= ((d-1)n,\dots,n,0).
$$Applying \eqref{Baapre!} to the choices $\bm p_n=\bm p+\bm n$ and $\bm q_n=\bm q+\bm n$, we obtain
\begin{align}\label{se-ki!}
\alpha_m=-\sum_{e\neq\sigma\in\Sigma_d} \operatorname{sgn}(\sigma)\alpha_{(\bm q_{\sigma}-\bm p) + (\bm n_\sigma -\bm n)}.
\end{align}Note that if $\sigma\neq e$, then at least one entry of $\bm n_\sigma -\bm n$ is non-zero and therefore for a fixed $\sigma\neq e$,
$$
\beta(\sigma,n):=\alpha_{(\bm q_{\sigma}-\bm p) + (\bm n_\sigma -\bm n)}
$$is a square-summable sequence. And therefore the linear combination
$$
\gamma_n:= -\sum_{e\neq\sigma\in\Sigma_d} \operatorname{sgn}(\sigma)\beta(\sigma,n)
$$must also be square-summable. But from \eqref{se-ki!} we see that for every $n\geq1$, $\gamma_n=\alpha_{\bm m}$. This implies that $\alpha_{\bm m}=0$ which was to prove.

\end{proof}
It is easy to see that the space $H^\infty(\mathbb G_d)$ consisting of all bounded analytic functions on $\mathbb G_d$ is contained in $H^2(\mathbb G_d)$. We identify $H^\infty(\mathbb G_d)$ with its boundary functions. In other words,
$$
H^\infty(\mathbb G_d)=\{\varphi\in L^\infty(b\Gamma_d): \varphi\circ \bm s \text{ has no negative Fourier coefficients}\}
$$
\begin{definition}
A Toeplitz operator with symbol $\varphi$ is called an {\em{analytic Toeplitz operator}} if $\varphi$ is in $H^\infty(\mathbb G_d)$.
A Toeplitz operator with symbol $\varphi$ is called a {\em{co-analytic Toeplitz operator}} if $T_\varphi^*$ is an analytic Toeplitz operator.
\end{definition}
Our next goal is to characterize analytic Toeplitz operators. But to be able to do that we need to define the following notion and the proposition following it.
\begin{definition}
Let $\varphi$ be in $L^\infty(b\Gamma_d)$. The operator $H_\varphi:H^2(\mathbb G_d)\to L^2(b\Gamma_d)\ominus H^2(\mathbb G_d)$ defined by
$$
H_\varphi f = (I -Pr)M_\varphi f
$$for all $f\in H^2(\mathbb G_d)$, is called a Hankel operator.
\end{definition}
We write down a few observations about Toeplitz operators some of which will be used in the theorem following it. The proofs are not written because they go along the same line as in the case of the unit disk.

\begin{proposition} \label{EasyObs} Let $\varphi \in L^\infty(b\Gamma_d)$. Then
\begin{enumerate}
\item $T_\varphi^* = T_{\overline{\varphi}}$.
\item If $\psi \in L^\infty(b\Gamma_d)$ is another function, then the product $T_\varphi T_\psi$ is another Toeplitz operator if $\overline{\varphi}$ or $\psi$ is analytic. In each case, $T_\varphi T_\psi = T_{\varphi \psi}$.
\item If $\psi \in L^\infty(b\Gamma_d)$, then $T_\varphi T_\psi - T_{\varphi \psi}=-H_{\overline{\varphi}}^*H_\psi.$
\item If $T_\varphi$ is compact, then $\varphi=0$.
\end{enumerate}
\end{proposition}
\subsection{Characterizations of analytic Toeplitz operators}
Now we are ready to characterize Toeplitz operators with analytic symbols. In the proof of this result, we make use of the following notation:
$$
[z]=\{\bm m=(m_1,m_2,\dots,m_d):m_1\geq m_2\geq\cdots\geq m_d\}.
$$Note the difference between $[z]$ and $\llbracket z \rrbracket$, defined in \S \ref{S:HardyG}. For an $\bm m \in[z]$, we denote by $\bm s_{\bm m}$, the symmetrization of the monomial $\bm z^{\bm m}$, i.e.,
$$
\bm s_{\bm m}(\bm z)=\sum_{\sigma\in \Sigma_d}\bm z^{\bm m_\sigma}.
$$It is easy to check that the set $\{\bm s_{\bm m}:\bm m \in [z]\}$ is an orthogonal basis of $L^2_{sym}(\mathbb T^d)$, the space consisting of the symmetric functions in $L^2(\mathbb T^d)$.
\begin{thm}\label{T:AnaToep}
Let $T_\varphi$ be a Toeplitz operator. Then the following are equivalent:
\begin{enumerate}
\item[(i)] $T_\varphi$ is an analytic Toeplitz operator;
\item[(ii)] $T_\varphi$ commutes with $T_p$;
\item[(iii)]$T_\varphi(Ran T_p)\subseteq Ran T_p$;
\item[(iv)] $T_pT_\varphi$ is a Toeplitz operator;
\item[(v)] For each $1\leq i \leq d-1$, $T_\varphi$ commutes with $T_{s_i}$;
\item[(vi)]For each $1\leq i \leq d-1$, $T_{s_i}T_\varphi$ is a Toeplitz operator.
\end{enumerate}
\end{thm}
\begin{proof}
\underline{$(i)\Leftrightarrow(ii)$:} The implication $(i)\Rightarrow (ii)$ is easy. To prove the other direction, we
note that if $T_\varphi$ commutes with $T_p$, then by part $(3)$ of Proposition \ref{EasyObs} we have $H_{\overline{p}}^*H_\varphi=0$. This shows that the corresponding product of Hankel operators on $H^2_{\text{anti}}(\mathbb D^d)$ is also zero, i.e.,
$H_{\overline{p}}^*H_{\varphi\circ \bm s}=0$. Let the power series expansion of $\varphi\circ \bm s \in L^\infty_\text{sym}(\mathbb T^d)$ be
$$
\varphi\circ \bm s(\bm z)=\sum_{\bm m \in \mathbb Z^d}\alpha_{\bm m} \bm z^m
$$with the convention that $\alpha_{\bm m_\sigma}=\alpha_{\bm m}$ for every $\sigma\in\Sigma_d$. Let $\bm p=(p_1,p_2,\dots,p_d)$ and $\bm q=(q_1,\dots,q_{d-1},0)$ be in $\llbracket n \rrbracket$. Then by a similar computation as done in \eqref{Baapre!} we have
\begin{eqnarray}
0=\langle H_{\varphi\circ \bm s} \bm {a_p}, H_{\overline{p}}\bm {a_q} \rangle_{L^2(\mathbb T^d)}&=&
\langle \sum_{\bm{m} \in [z]}\alpha_{\bm m} \bm z^{\bm m}\bm {a_p}, \bm a_{\bm q'}\rangle_{L^2(\mathbb T^d)}\;\;\;[\bm q':=\bm q-(1,1,\dots,1)]\notag\\
&=& d!\sum_{\sigma\in\Sigma_d}\operatorname{sgn}(\sigma)\alpha_{\bm q'_{\sigma}-\bm p}.\label{Baapre2!}
\end{eqnarray}
Since the last entry of $\bm q'$ is $-1$, at least one entry of $\bm q'_\sigma -\bm p$ is negative for every $\bm p \in\llbracket n\rrbracket$ and $\sigma\in\Sigma_d$. Now following a similar line of analysis done in the proof of Lemma \ref{L:Injection}, one can conclude from \ref{Baapre2!} that $\alpha_{\bm m}=0$ whenever $\bm m \in [z]$ with the last entry negative. Hence by the symmetric property of $\varphi\circ\bm s$, $\alpha_{\bm m}=0$, if any coordinates of $\bm m$ is negative, in other words, $\varphi$ is analytic in $\mathbb G_d$.

\underline{$(ii)\Leftrightarrow (iii)$:} The part $(ii)\Rightarrow(iii)$ is easy. Conversely, suppose that $RanT_p$ is invariant under $T_\varphi$. Since $RanT_p$ is closed, we have for every $f\in H^2(\mathbb G_d)$,
\begin{eqnarray*}
&&T_\varphi T_p f=T_p g_f \text{ for some $g_f$ in $H^2(\mathbb G_d)$.}
\\
&\Rightarrow& T_p^*T_\varphi T_p f=g_f \Rightarrow T_\varphi f =g_f \;(\text{by the Brown--Halmos relations \eqref{Toeplitzcharc}}).
\end{eqnarray*} Hence $T_\varphi T_p=T_p T_\varphi$.

\underline{$(ii)\Leftrightarrow (iv)$:} If $T_\varphi$ commutes with $T_p$, then $T_pT_\varphi$ is same as $T_\varphi T_p$, which is a Toeplitz operator by Proposition \ref{EasyObs}. Conversely, if $T_pT_\varphi$ is a Toeplitz operator, then it satisfies Brown--Halmos relations, the second one of which implies that $T_\varphi$ commutes with $T_p$.

\underline{$(i)\Leftrightarrow(v)$:} For an analytic symbol $\varphi$, $T_\varphi$ obviously commutes with each $T_{s_i}$. The proof of the converse direction is done by the same technique as in the proof of $(ii)\Rightarrow (i)$. If $T_\varphi$ commutes with $T_{s_i}$, then by part $(3)$ of Proposition \ref{EasyObs} we have $H_{\overline{s_i}}^*H_\varphi=0$. Suppose $\varphi\circ \bm s \in L^\infty_\text{sym}(\mathbb T^d)$ has the following power series expansion
$$\varphi\circ \bm s(\bm z)=\sum_{\bm m \in\mathbb Z^d}\alpha_{\bm m} \bm z^{\bm m}.$$
Let $\bm q=(q_1,\dots,q_{d-1},0)$ and $\bm p=(p_1,p_2,\dots,p_d)$ be in $\llbracket n \rrbracket$ such that $q_{j}-q_{j-1}=1$ for every $2\leq j \leq d-1$. Note that for this choice of $\bm q$, the only non analytic term in the expression of $\overline{s_i}\bm{a_q}(\bm z)$ is $\bm {a_{q'}}$, where
 $$
\bm q'=\bm q-(0,\dots,0,\underbrace{1,\dots,1}_{i-\text{times}}).
 $$A similar computation as in \eqref{Baapre!} yields
\begin{eqnarray*}
0=\langle H_{\varphi\circ \bm s} \bm{a_p}, H_{\overline{s_i}}\bm{a_q} \rangle_{L^2(\mathbb T^d)}
=\langle H_{\varphi\circ \bm s} \bm{a_p}, \bm{a_{q'}} \rangle_{L^2(\mathbb T^d)}
= d!\sum_{\sigma\in\Sigma_d}\operatorname{sgn}(\sigma)\alpha_{\bm q'_{\sigma}-\bm p}. \label{Baapre3!}
\end{eqnarray*}
By a similar analysis as in the proof of $(ii)\Rightarrow(i)$ above, one can conclude that $\alpha_{\bm m}=0$ whenever $\bm m$ has at least one negative entry. This proves that $\varphi$ is analytic.\\
\underline{$(v)\Leftrightarrow (vi)$:} The implication $(v)\Rightarrow (vii)$ follows from Proposition \ref{EasyObs}. Conversely suppose that $T_{s_i}T_\varphi$ is a Toeplitz operator. Therefore applying Theorem \ref{T:BH} and Lemma \ref{L:GammaUni}, we get $T_\varphi T_{s_i}={T^*_{s_{d-i}}}T_\varphi T_p={T_p}^*T_{s_i}T_\varphi T_p=T_{s_i}T_\varphi.$
\end{proof}
\begin{remark}
Since adjoint of an analytic Toeplitz operator is a co-analytic Toeplitz operator and vice versa, Theorem \ref{T:AnaToep} in turn characterizes co-analytic Toeplitz operators as well.
\end{remark}
%
%

\section{Compact perturbation of Toeplitz operators}\label{S:asympT}
In this section we first find a characterization of compact operators on $H^2(\mathbb G_d)$ and then use it to characterize compact perturbation of Toeplitz operators -- the so called asymptotic Toeplitz operators.

Note that for a bounded operator $B$ on $H^2(\mathbb D)$, $B$ is a Toeplitz operator if and only if there exists $T$ on $H^2(\mathbb D)$ such that ${T_z^*}^nTT_z^n\to B$ weakly on $H^2(\mathbb D)$. Although $T_p$ of $H^2(\mathbb G_d)$ is unitarily equivalent to $T_z$ on $H^2_\mathcal E(\mathbb D)$, a similar convergence like above does not guarantee that the limit is a Toeplitz operator, see Example \ref{FO}. A necessary and sufficient condition for the limit to be Toeplitz is given below.
\begin{lemma}\label{cond-limit-Toepitz}
Let $T$ and $B$ be bounded operators on $H^2(\mathbb G_d)$ such that ${T_p^*}^nTT_p^n\to B$ weakly. Then $B$ is Toeplitz operator if and only if for each $1\leq i\leq d-1$
$$
{T_p^*}^n[T,T_{s_i}]T_p^n \to 0 \text{ weakly,}
$$
where $[T,T_{s_i}]$ denotes the commutator of $T$ and $T_{s_i}$.
\end{lemma}
\begin{proof}
Note that if $T$ and $B$ are bounded operators on $H^2(\mathbb G_d)$ such that ${T_p^*}^nTT_p^n\to B$ weakly, then $T_p^*BT_p=B$. Suppose ${T_p^*}^n[T,T_s]T_p^n \to 0$ weakly. To prove that $B$ is Toeplitz, it remains to show that $B$ satisfies the other Brown--Halmos relations with respect to the $\Gamma_d$-isometry $(T_{s_1},\dots,T_{s_{d-1}}, T_p)$. For any $i$ with $1\le i\le d-1$
\begin{eqnarray*}
&&T_{s_i}^*BT_p=\text{w-}\!\lim T_{s_i}^*({T_p^*}^nTT_p^n)T_p=\text{w-}\!\lim {T_p^*}^n(T_{s_i}^*TT_p)T_p^n=\text{w-}\!\lim {T_p^*}^{n+1}T_{s_{d-i}}TT_p^{n+1}\\
&=&\text{w-}\!\lim {T_p^*}^{n+1}(T_{s_{d-i}}T-TT_{s_{d-i}}+TT_{s_{d-i}})T_p^{n+1}=\text{w-}\!\lim {T_p^*}^{n+1}TT_p^{n+1}T_{s_{d-i}}=BT_{s_{d-i}}.
\end{eqnarray*}
Conversely, suppose that the weak limit $B$ of ${T_p^*}^nTT_p^n$ is a Toeplitz operator and hence satisfies the Brown--Halmos relations. Thus,
\begin{eqnarray*}
&&\text{w-}\!\lim{T_p^*}^n(TT_{s_i}-T_{s_i}T)T_p^n=\text{w-}\!\lim({T_p^*}^nTT_p^nT_{s_i}-{T_p^*}^nT_{s_{d-i}}^*T_pTT_p^n)
\\
&=&\text{w-}\!\lim({T_p^*}^nTT_p^nT_{s_i}-T_{s_{d-i}}^*{T_p^*}^{d-1}TT_p^{d-1}T_p)=BT_{s_i}-T_{s_{d-i}}^*BT_p=0,
\end{eqnarray*}
for all $1\le i\le d-1$.
\end{proof}
Before we characterize the compact operators on $H^2(\mathbb G_d)$, we recall the analogous result for the polydisk, discovered in \cite{MSS}. The one dimensional case was proved by Feintuch \cite{Feintuch}.

For $m\geq1$, define a completely positive map $\eta_m : \mathcal B ( H^2(\mathbb D^d) \rightarrow \mathcal B (\oplus_{i=1}^n H^2(\mathbb D^d))$ by
$$ \eta_m(T) =\begin{bmatrix}
               T_{z_1}^{* m} \\
               T_{z_2}^{* m} \\
               \vdots \\
               T_{z_d}^{* m}
             \end{bmatrix} T \begin{bmatrix}
                               T_{z_1}^m, & T_{z_2}^m, & \cdots,  & T_{z_d}^m \\
                             \end{bmatrix}
                           .$$
\begin{thm}
A bounded operator $T$ on $H^2(\mathbb D^d)$ is compact if and only if $\eta_m(T) \to 0$ in norm as $m\to \infty$.
\end{thm}
This shows the importance of the forward shifts in characterizing the compact operators.

For $1\leq j \leq d-1$, let $Y_j$ be the operator on $H^2_{\text{anti}}(\mathbb D^d)$ as defined in (\ref{the-other-shift}). Let $X_j$ be the restriction of $Y_j$ to $\mathcal E$, the coefficient space of the symmetrized polydisk, i.e., for every $\bm p \in \llbracket n \rrbracket$ such that $\bm p=(p_1,\dots,p_{d-1},0)$,
$$
X_j\bm{a_p}=\bm{{a_{p+f_j}}}, \text{ where } \bm{f_j}=(\underbrace{1,\dots,1}_{j\text{-times}},0\dots,0) .
$$Then each $Y_j$ on $H^2_{\text{anti}}(\mathbb D^d)$ has the following expression:
$$
Y_j=\sum_{r=0}^\infty T_{p}^rX_jT_{p}^{* r}.
$$It should be noted that each $Y_j$ is a pure isometry and the set of operators $\{Y_1,Y_2,\dots,Y_{d-1}\}$ is doubly commuting.
Consider the following operator on $\mathcal E$
$$E_l:=(P_\mathcal E-X_1^{l}X_1^{* l})(P_\mathcal E-X_2^lX_2^{* l})\cdots (P_\mathcal E-X_{d-1}^lX_{d-1}^{* l}).$$
Since the operators $X_1,X_2,\cdots,X_{d-1}$ are doubly commuting, the operator $E_l$ is an orthogonal projection of $\mathcal E$ onto the $l$ dimensional space
 $$
\bigcap_{j=1}^{d-1}\ker X_{j}^{* l}= \text{span}\{\bm a_{\bm p}: \bm p=\big(k+(d-2)l,k+(d-3)l,\dots,k,0\big),\;1\leq k \leq l\}.
 $$ Also note that
 \begin{eqnarray*}
 &&E_l =\\
 &&P_\mathcal E - \sum_{j=1}^{d-1}X_{j}^lX_{j}^{* l}+\sum_{1\leq j_1<j_2\leq d-1}X_{j_1}^lX_{j_2}^lX_{j_1}^{* l}X_{j_2}^{* l}+\\&&\cdots+(-1)^{k}\sum_{1\leq j_1<\cdots<j_k\leq d-1}X_{j_1}^l\cdots X_{j_k}^lX_{j_1}^{* l}\cdots X_{j_k}^{* l}+\cdots+(-1)^{d-1}X_{1}^l\cdots X_{n}^lX_{1}^{* l}\cdots X_{n}^{* l}.
 \end{eqnarray*}
We have the following characterization of compact operators on $H^2(\mathbb G_d)$ using the set of pure isometries $\{Y_1,\dots,Y_{d-1},T_p\}$.
\begin{thm}\label{charc-compact}
For $j\geq1$, define completely positive maps $$\eta_j : \mathcal B (H^2(\mathbb G_d)) \to \mathcal B (\oplus_{i=1}^n H^2(\mathbb G_d))$$ by
$$ \eta_j (T) = \begin{bmatrix}
                    Z_1^{*j} \\
                    \vdots\\
                    Z_{d-1}^{*j} \\
                    T_p^{*j} \\
                  \end{bmatrix}T
                           \begin{bmatrix}
                             Z_1^j \cdots Z_{d-1}^j  T_p^n \\
                           \end{bmatrix},$$
where $Z_l = U^* Y_lU$. A bounded operator $T$ on $H^2(\mathbb G_d)$ is compact if and only if $\eta_j(T)\to 0$ in norm as $j\to\infty$.
\end{thm}
\begin{proof}
The necessity follows from a straightforward application of Lemma 3.1 of \cite{MSS}, which states that if $P$ is a pure contraction, $T$ is a contraction and $K$ is a compact operator on a Hilbert space, then $P^{* l}KT \to 0$ in the operator norm as $l\to \infty$.

Conversely, suppose that $T$ is a bounded operator on $H^2(\mathbb G_d)$ satisfying the convergence condition. We shall conclude by finding a finite rank operators approximation of $T$. To that end note that $U^*P_{\mathcal E}U=U^*(I-T_{p}T_{p}^*)U=I-T_pT_p^*$ and
 \begin{eqnarray}\label{propofU}
 U^*E_lU=(I-T_pT_p^*)-\sum_{k=1}^{d-1}(-1)^{k}\sum_{1\leq j_1<\cdots<j_k\leq d-1}W_{j_1}^l\cdots W_{j_k}^lW_{j_1}^{* l}\cdots W_{j_k}^{* l},
 \end{eqnarray}where $W_j=U^*X_jU$. For each positive integer $l$, consider the following finite rank operators on $H^2(\mathbb G_d)$:
 $$
 F_l=U^*\left(E_l+T_{p}E_lT_{p}^*+\cdots +T_{p}^{l-1}E_lT_{p}^{* l-1}\right)U.
 $$
Using (\ref{propofU}) we get $F_l$ to be the same as
\begin{align*}
&(I-T_p^lT_p^{* l})-\sum_{r=0}^{l-1}T_p^r\left(\sum_{j=1}^{d-1}W_j^lW_j^{* l}\right)T_p^{* r}+\cdots\\&+(-1)^{k}\sum_{r=0}^{l-1}T_p^r\left(\sum_{1\leq j_1<\cdots<j_k\leq d-1}W_{j_1}^l\cdots W_{j_k}^lW_{j_1}^{* l}\cdots W_{j_k}^{* l}\right)T_p^{* r}+\cdots\\&+(-1)^{d-1}\sum_{r=0}^{l-1}T_p^r\left(W_{1}^l\cdots W_{n}^lW_{1}^{* l}\cdots W_{n}^{* l}\right)T_p^{* r}.
\end{align*}
Let $P_l$ be the projection of $H^2(\mathbb G_d)$ onto the space $\ker T_p^{* l-1}$. Since $Z_j=\sum_{r=0}^\infty T_p^rW_jT_p^{* r}$,
\begin{eqnarray*}
I-F_l&=&{T_p}^l{T_p}^{* l} + \sum_{j=1}^{d-1}P_l Z_j^lZ_j^{* l}P_l+\cdots+(-1)^k\sum_{1\leq j_1<\cdots<j_k\leq d-1}P_lZ_{j_1}^l\cdots Z_{j_k}^lZ_{j_1}^{* l}\cdots Z_{j_k}^{* l}P_l\\&&+(-1)^nP_lZ_{1}^l\cdots Z_{d-1}^lZ_{1}^{* l}\cdots Z_{d-1}^{* l}P_l.
\end{eqnarray*}
Then the operator $\tilde F_l=TF_l+F_lT-F_lTF_l$ is also a finite rank operator and note that $T-\tilde F_l=(I-F_l)T(I-F_l)$.
By the above form of $(I-F_l)$ and from the hypotheses it follows that $\|T-\tilde F_l\|\to 0$. Hence, $T$ is compact.
\end{proof}
\begin{definition}
A bounded operator $T$ on $H^2(\mathbb G_d)$ is called an asymptotic Toeplitz operator if $T=T_\varphi +K$ for some $\varphi \in L^\infty (b\Gamma)$ and compact operator $K$ on $H^2(\mathbb G_d)$.
\end{definition}
We end this section with the following characterization of asymptotic Toeplitz operators.
\begin{thm}
A bounded operator $T$ on $H^2(\mathbb G_d)$ is an asymptotic Toeplitz operator if and only if ${T_p^*}^n[T,T_s]T_p^n\to 0$, ${T_p^*}^nTT_p^n\to B$ and $\eta_d(T-B) \to 0$.
\end{thm}
\begin{proof}
If $T$ satisfies the convergence conditions as in the statement, then it follows from Lemma \ref{cond-limit-Toepitz} that $B$ is a Toeplitz operator because ${T_p^*}^n[T,T_s]T_p^n\to 0$. Also, since $\eta_d(T-B)\to 0$, by Theorem \ref{charc-compact}, $T-B$ is a compact operator. Hence $T$ is the sum of a compact operator and a Toeplitz operator.

Conversely, suppose $T$ is an asymptotic Toeplitz operator, i.e., $T=K+T_\varphi$, where $K$ is some compact operator. Then by Theorem \ref{charc-compact}, ${T_p^*}^nTT_p^n\to T_\varphi$. Since $T_\varphi$ is Toeplitz, by Lemma \ref{cond-limit-Toepitz}, ${T_p^*}^n[T,T_s]T_p^n\to 0$. And finally, since $K$ is compact, by Theorem \ref{charc-compact}, $\eta_d(T-T_\varphi)\to 0$.
\end{proof}

\section{Generalized Toeplitz operators}\label{S:Gen}
In this section we briefly describe the operator theory of the symmetrized polydisk. This was introduced in \cite{SS}. We then extend the notion of Toeplitz operators to what we call an $\underline{S}$-Toeplitz operator, where a commuting $d$-tuple $\underline{S}$ is an `isometry' associated to the symmetrized polydisk, to be formally defined below.

\begin{definition}
A $d$-tuple $\underline{T}=(T_1,\dots,T_{d-1},P)$ of commuting bounded operators on a Hilbert space $\mathcal H$ is called a {\em $\Gamma_d$-contraction}, if $\Gamma_d$ is a spectral set for $\underline{T}$, i.e., (since $\Gamma_d$ is polynomially convex) for every polynomial $f$ in $d$ commuting variables
$$
\|f(T_1,\dots,T_{d-1},P)\|\leq \sup_{\bm z\in\Gamma_d}|f(\bm z)|=:\|f\|_{\infty,\Gamma_d}.
$$
\end{definition}
It is clear from the definition that $(T_1,\dots,T_{d-1},P)$ is a $\Gamma_d$-contraction if and only if $(T_1^*,\dots,T_{d-1}^*,P^*)$ is so. It is also clear that the last component $P$ of a $\Gamma_d$-contraction is a contraction, which follows by applying the definition of a $\Gamma_d$-contraction to the particular function that projects $\mathbb C^d$ to the last component. A profound study of $\Gamma_2$-contraction -- initiated by Agler and Young -- is what prompted the study of such a tuple. There are natural $\Gamma_d$-analogues of unitaries and isometries.
\begin{definition}
A $d$-tuple $\underline{R}=(R_1,\dots,R_{d-1},U)$ of commuting bounded normal operators on a Hilbert space is said to be a {\em $\Gamma_d$-unitary}, if the Taylor joint spectrum of $\underline{R}$, $\sigma(\underline{R})$ is contained in the distinguished boundary of $\Gamma_d$.

A {\em$\Gamma_d$-isometry} is a $d$-tuple $\underline{S}=(S_1,\dots,S_{d-1},V)$ of commuting bounded operators that has a $\Gamma_d$-unitary extension.
\end{definition}
Let us recall that the projection of the distinguished boundary of $\Gamma_d$ to the last component is $\mathbb T$. Therefore by definition, the last entry of a $\Gamma_d$-isometry and a $\Gamma_d$-unitary are an isometry and a unitary, respectively.
The following characterizations of $\Gamma_d$-unitaries and $\Gamma_d$-isometries will be used.
\begin{thm}[Theorem 4.2 \cite{SS}] \label{G-unitary}
Let $(R_1,\dots,R_{d-1},U)$ be a pair of commuting operators defined on a Hilbert
space $\mathcal{H}.$ Then the following are equivalent:
\begin{enumerate}
\item $(R_1,\dots,R_{d-1},U)$ is a $\Gamma_d$-unitary;
\item there exist commuting unitary operators $U_{1},U_2\dots,U_d$ on $\mathcal{H}$ such that for $1\leq i \leq d$,
$$R_i= \sum_{1\leq k_1<\cdots<k_i\leq d}U_{k_1}U_{k_2}\cdots U_{k_i};$$
\item $U$ is unitary, $R_{d-i}=R_{i}^*U,$ and $(\gamma_1R_1,\gamma_2R_2,\dots,\gamma_{d-1}R_{d-1})$ is a $\Gamma_{d-1}$-contraction where $\gamma_i=\frac{d-i}{d}$ for $i=1,2,\dots,d-1$.
\end{enumerate}
\end{thm}
\begin{thm}[Theorem 4.12 \cite{SS}] \label{G-isometry}
Let $(S_1,\dots,S_{d-1},V)$ be a pair of commuting operators defined on a Hilbert
space $\mathcal{H}.$ Then the following are equivalent:
\begin{enumerate}
\item $(S_1,\dots,S_{d-1},V)$ is a $\Gamma_d$-isometry;
\item $V$ is isometry, $S_{d-i}=S_{i}^*V,$ and $(\gamma_1S_1,\gamma_2S_2,\dots,\gamma_{d-1}S_{d-1})$ is a $\Gamma_{d-1}$-contraction where $\gamma_i=\frac{d-i}{d}$ for $i=1,2,\dots,d-1$.
\end{enumerate}
\end{thm}
The tuple $(T_{s_1},\dots,T_{s_{d-1}},T_p)$ of multiplication by coordinate functions on $H^2(\mathbb G_d)$ is a prototype example of a $\Gamma_d$-isometry. By Theorem \ref{T:BH}, Toeplitz operators on $H^2(\mathbb G)$ are precisely those that satisfy the Brown--Halmos relations
$$
T_{s_i}^*TT_p=TT_{s_{d-i}} \text{ and }T_p^*TT_p=T \text{ for each }i=1,2\dots,d-1.
$$
This motivates us to consider the following more general situation.
\begin{definition}
Given a $\Gamma_d$-isometry $\underline{S}=(S_1,\cdots,S_{d-1},V)$ acting on a Hilbert space $\cH$, a bounded operator $X$ on $\cH$ is called an {\em$\underline{S}$-Toeplitz operator} if it satisfies the Brown--Halmos relation with respect to the $\Gamma_d$-isometry $\underline{S}$, i.e.,
$$
{S_i}^*XV=XS_{d-i} \text{ and }V^*XV=X \text{ for each }i=1,2\dots,d-1.
$$
\end{definition}The following is the main result of this section.
\begin{thm}\label{the-iso-case}
Let $\underline{S}=(S_1,\dots,S_{d-1},V)$ be a $\Gamma_d$-isometry on a Hilbert space
$\mathcal H$. Then
\begin{enumerate}
\item[(1)]\textbf{Extension:} There exists a  $\Gamma_d$-unitary $\underline{R}=(R_1,\dots,R_{d-1},U)$ acting on
a Hilbert space $\mathcal K$ containing $\mathcal H$ such that $\underline{R}$ is the minimal extension of
$\underline{S}$. In fact,
\begin{align}\label{miniK}
\mathcal K=\overline{\operatorname{span}}\{U^{ m}h:h\in\mathcal H, \text{ and } m \in \mathbb Z\}.
\end{align}
    Moreover, any operator $X$ acting on $\mathcal H$ commutes with $\underline{S}$ if and only if
    $X$ has a unique norm preserving extension $Y$ acting on $\mathcal K$ commuting with $\underline{R}$;
\item[(2)]\textbf{Halmos dilation:} An operator $X$ is an $\underline{S}$-Toeplitz operator if and only if there exists a unique norm-preserving Halmos dilation of $X$ in $\{R_1,\dots,R_{d-1},U\}'$, i.e., there exists a unique operator $Y$ in $\{R_1,\dots,R_{d-1},U\}'$ such that
    \begin{align}\label{uniqueness}
    \|X\|=\|Y\|\text{ and }  X=P_\mathcal HY|_{\mathcal H};
    \end{align}
\item[(3)]\textbf{Extension of $C^*(\underline{R})$:} Let $\mathcal C^*(\underline{S})$ and $\mathcal C^*(\underline{R})$ denote the unital $\mathcal C^*$-algebras generated by
$\{S_1,\dots,S_{d-1},V\}$ and $\{R_1,\dots,R_{d-1},U\}$, respectively and $\mathcal I(\underline{S})$
denote the closed ideal of $\mathcal C^*(\underline{S})$ generated by all the commutators $XY-YX$ for
 $X,Y\in \mathcal C^*(\underline{S})\cap \mathcal T(\underline{S})$.
Then there exists a short exact sequence
    $$
    0\rightarrow\mathcal I(\underline{S})\hookrightarrow \mathcal C^*(\underline{S})\xrightarrow{\pi_0} \mathcal C^*(\underline{R})\rightarrow 0
    $$
    with a completely isometric cross section, where $\pi_0: \mathcal C^*(\underline{S})\to \mathcal C^*(\underline{R})$ is the canonical unital $*$-homomorphism which sends the generating set $\underline{S}$
    to the corresponding generating set $\underline{R}$, i.e., $\pi_0(V)=U$ and $\pi_0(S_i)=R_i$ for all $1\leq i \leq d-1$.
\end{enumerate}
\end{thm}
\begin{proof}
The main idea is to apply Theorem 1.2 of Prunaru \cite{Prunaru} to the isometry part $V$ of the $\Gamma_d$-isometry $\underline{S}=(S_1,\dots,S_{d-1},V)$  and then to progress with a similar line of analysis as in Prunaru's proof of his theorem for spherical isometries to conclude the proof of the theorem.

The starting point is the completely positive and unital map $\Phi_0:\cB(\cH)\to\cB(\cH)$ defined as
$$
\Phi_0:X\to V^*XV.
$$We note that $\Phi_0$ is a normal mapping, i.e., $\Phi_0(A_i)\to \Phi_0(A)$ in the strong operator topology for every net $\{A_i\}_{i\in I}$ of increasing operators that goes to $A$ in the strong operator topology. Thus, the map $\Phi_0$ satisfies all the hypotheses of Lemma 2.3 of \cite{Prunaru}, a simpler version of which states that {\em if $\Phi_0$ is a completely positive, unital and normal mapping on $\cB(\cH)$, then there exists an idempotent completely positive map $\Phi$ on $\cB(\cH)$ such that
$$
\cT(V):=\{X:\Phi_0(X)=X\}=\operatorname{Ran}\Phi.
$$}
Now using a well-known result of Choi--Effros \cite{CE}
for the idempotent unital completely positive map $\Phi$, one has
\begin{equation}
\label{ideal}
 \Phi(\Phi(X)Y)=\Phi(X\Phi(Y))=\Phi(\Phi(X)\Phi(Y))\quad
 (X,Y\in \mathcal B(\cH)).
\end{equation}
  Let us denote by $\Phi_\nu$ the restriction of $\Phi$ to $C^*(\cT(V))$, the $C^*$-algebra generated by $\cT(V)$.
Then it readily follows from \eqref{ideal} that $\operatorname{Ker}\Phi_\nu$ is an ideal.
  Let $(\Pi,\pi,\cK)$ denote the minimal Stinespring dilation of $\Phi_\nu$, i.e., $\Pi:\cH\to\cK$ is an isometry (since $\Phi_\nu$ is unital) and
\begin{align}\label{Stines}
\Phi_\nu(X)=\Pi^*\pi(X)\Pi \quad \text{ for all $X\in C^*(\cT(V))$}.
\end{align}As a consequence of \eqref{ideal} and the minimality of the Stinespring dilation $(\Pi,\pi,\cK)$, we get
$$
\operatorname{Ker}\Phi_\nu=\operatorname{Ker}\pi.
$$
Indeed, from \eqref{Stines} the inclusion $\operatorname{Ker}\Phi_\nu\supset\operatorname{Ker}\pi$ is clear, while for the other inclusion we note using \eqref{Stines} that for any $X,Y,Z$ in $C^*(\cT(V))$ and $h,k$ in $\cH$
\begin{align*}
\langle \Phi_\nu(X^*YZ)h,k\rangle=\langle  \Pi^*\pi(X^*YZ)\Pi h,k\rangle =\langle \pi(Y) \pi(Z)\Pi(h),\pi(X)\Pi(k) \rangle.
\end{align*}This implies that if $Y$ in $C^*(\cT(V))$ is such that $\Phi_\nu(Y)=0$, then since the kernel of $\Phi_\nu$ is an ideal, $\Phi_\nu(XYZ)=0$ and hence $\pi(Y)=0$ because by minimality of the dilation
$$
\overline{\operatorname{span}}\{\pi(X)\Pi h: X\in C^*(\cT(V)),\;h\in\cH\}=\cK.
$$
Let us now define the mapping $\rho:\pi(C^*(\cT(V)))\to\cB(\cH)$ as
\begin{align}\label{rho}
\rho(\pi(X))=\Pi^*\pi(X)\Pi=\Phi_\nu(X).
\end{align}Since $\Phi_\nu$ is idempotent, it follows that for every $X$ in $C^*(\cT(V))$, $\Phi_\nu$ kills $X-\Phi_\nu(X)$ and therefore so does $\pi$. Using this fact, one easily deduces that
$$
\pi\circ\rho=I_{\operatorname{Ran}\pi}
$$
and the mapping $\rho$ is a complete isometry.
We now note down two further crucial properties of the Stinespring dilation $(\Pi,\pi,\cK)$ which can be derived from what we discussed above. This can also be seen from the proof of Theorem 1.2 in \cite{Prunaru} when applied to the case of a single isometry $V$.
\begin{enumerate}
\item[($\textbf P_1$)] The operator
$$
U:=\pi(V),
$$ is the minimal unitary extension of $V$, i.e., $\cK$ is as in \eqref{miniK};
\item[\bf($\textbf P_2$)] if $X$ commutes with $V$, then $\pi(X)$ is the unique operator that commutes with $U$, $\|\pi(X)\|=\|X\|$ and $\pi(X)|_{\Pi(\cH)}=\Pi X\Pi^*|_{\Pi(\cH)}$;
\end{enumerate}{\em We shall identify $\cH$ with $\Pi \cH$ and therefore view $\cH$ as a subspace of $\cK$.} To prove item (1) of the theorem, we proceed by defining
$$
R_j=\pi(S_j)\quad \text{for each } j=1,2,\dots, d-1.
$$
Commutativity of the tuple $(R_1,\dots,R_{d-1},U)$ follows from that of $(S_1,\dots,S_{d-1},V)$. We shall use item (3) of Theorem \ref{G-unitary} to prove that $\underline{R}=(R_1,\dots,R_{d-1},U)$ is a $\Gamma_d$-unitary. By $({\bf{ P_1}})$ we know that $U$ is a unitary. We also see that for each $j=1,2,\dots,d-1$
$$R^*_j U =\pi(S_j^*) \pi(V)=\pi(S_j^*V)=\pi(S_{d-j})=R_{d-j},$$where we used item (2) of Theorem \ref{G-isometry}. Now to see that
$$(\frac{d-1}{d}R_1, \frac{d-2}{d}R_2, \ldots , \frac{1}{d}R_{d-1})$$
is a $\Gamma_{d-1}$-contraction, we choose any polynomial $f$ in $d-1$ variables, and compute
 \begin{align*} &\| f(\frac{d-1}{d}R_1, \frac{d-2}{d}R_2, \ldots , \frac{1}{d}R_{d-1})\| \\
 = &\| f(\frac{d-1}{d}\pi(S_1), \frac{d-2}{d}\pi(S_2), \ldots , \frac{1}{d}\pi(S_{d-1})) \| \\
 = &\| \pi ( f(\frac{d-1}{d}S_1, \frac{d-2}{d}S_2, \ldots , \frac{1}{d}S_{d-1}) )\| \quad[\mbox{because } \pi \mbox{ is multiplicative}]\\
 \le &\|  f(\frac{d-1}{d}S_1, \frac{d-2}{d}S_2, \ldots , \frac{1}{d}S_{d-1} )\| \quad[\mbox{because } \pi \mbox{ is contractive}].\end{align*}
 Now we use Lemma 2.7 of \cite{SS} which states that $(\frac{d-1}{d}S_1, \frac{d-2}{d}S_2, \ldots , \frac{1}{d}S_{d-1})$ is a $\Gamma_{d-1}$-contraction whenever $(S_1, \ldots ,S_{d-1}, P)$ is a $\Gamma_d$-contraction to conclude that the last line of the displayed equations above is dominated by the supremum of $|f|$ over the set $\Gamma_{d-1}$. Consequently, we have shown that $(\frac{d-1}{d}R_1, \frac{d-2}{d}R_2, \ldots , \frac{1}{d}R_{d-1})$ is a $\Gamma_{d-1}$-contraction.

The minimality of the extension $\underline{R}=(R_1,\dots,R_{d-1},U)$ follows from $({\bf{P_1}})$.

For the moreover part in item (1), let $X$ be an operator on $\mathcal H$ which commutes with $\underline{S}$. Set $Y:=\pi(X)$. By $({\bf{P_2}})$ above, we know that $Y$ is the unique operator that commutes with $U$ and is a norm preserving extension of $X$. Furthermore, $Y$ commutes with each of $R_j$ as $X$ commutes with each of $S_j$. Conversely, if $X$ in $\cB(\cH)$ has an extension that commutes with $(R_1,\dots,R_{d-1},U)$, then clearly $X$ would commute with $(S_1,\dots,S_{d-1},V)$ because $(R_1,\dots,R_{d-1},U)$ is an extension of $(S_1,\dots,S_{d-1},V)$. This completes the proof of part (1) of this theorem.

 For part (2), let $X$ be an $\underline{S}$-Toeplitz operator and as before set $Y:=\pi(X)$. By definition we have
 $$
{S_i}^*XV=XS_{d-i} \text{ and }V^*XV=X \text{ for each }i=1,2\dots,d-1,
$$applying $\pi$ on both sides of which we get
$$
{R_i}^*YU=YR_{d-i} \text{ and }U^*YU=Y \text{ for each }i=1,2\dots,d-1.
$$The last equation is equivalent to $YU=UY$, which, in view of the other equations,
further implies that $YR_j=R_jY$ for each $i=1,2\dots,d-1$. For the norm equality we simply observe that $X$ is the compression of $Y$ to $\cH$ and therefore  $\|X\|\leq \|Y\|=\|\pi(X)\|\leq \|X\|.$ For the uniqueness part we use $({\bf{P_3}})$. Suppose that there are two possibly distinct operators $Y_1$ and $Y_2$ that satisfy \eqref{uniqueness}. Then
$$
0=P_{\cH}(Y_1-Y_2)|_{\cH}=\rho(Y_1-Y_2).
$$Since $\rho$ is an isometry, $Y_1=Y_2$.

For the converse direction of (2), let us first note that {\em commutativity is much stronger than being a generalized Toeplitz operator, i.e., if $X$ commutes with $\underline{S}$, then $X$ trivially becomes an $\underline{S}$-Toeplitz operator.} Therefore if $Y$ commutes with $(R_1,\dots,R_{d-1},U)$, then in particular we have
\begin{align}\label{aux1}
R_j^* YU=YR_{d-j} \text{ and }U^*YU=U  \text{ for each }j=1,2\dots,d-1.
\end{align}Moreover we have the following $2\times2$ matrix structure of the concerned operators:
$$
R_j=\left[
                 \begin{array}{cc}
                   S_j & * \\
                   0 & * \\
                 \end{array}
               \right], U=\left[
                 \begin{array}{cc}
                   V & * \\
                   0 & * \\
                 \end{array}
               \right]\text{ and }Y=\left[
                 \begin{array}{cc}
                   X & * \\
                   * & * \\
                 \end{array}
               \right].
$$ Using these matrix structures and \eqref{aux1}, one readily obtains that $X$ is an $\underline{S}$-Toeplitz operator. This completes the proof of part (2).

 To prove part (3), we first note that the representation $\pi_0$ in the statement of the theorem is actually the restriction of $\pi$ to $\mathcal C^*(\underline{S})$ as the representation $\pi$
also maps the generating set $\underline{S}$ of $\mathcal C^*(\underline{S})$ to the generating set $\underline{R}$ of $\mathcal C^*(\underline{R})$.
Since $\pi_0(\underline{S})=\underline{R}$, range of $\pi_0$ is $\mathcal C^*(\underline{R})$.
Therefore to prove that the following sequence
$$
0\rightarrow\mathcal I(\underline{S})\hookrightarrow \mathcal C^*(\underline{S})\xrightarrow{\pi_0} \mathcal C^*(\underline{R})\rightarrow 0
$$
is a short exact sequence, all we need to show is that ker$\pi_0=\mathcal I(\underline{S})$.

We do that now. Since $\pi_0(\mathcal C^*(\underline{S}))$ is commutative, we have $XY-YX$ in the kernel of $\pi_0$, for any $X,Y\in \mathcal C^*(\underline{S})\cap \mathcal T(\underline{S})$.
Hence $\mathcal I(\underline{S})\subseteq$ ker$\pi_0$. To prove the other inclusion, consider a finite product $Z_1$ of members of $\underline{S}^*=(S_1^*,\dots,S_{d-1}^*,V^*)$ and a finite product $Z_2$ of members of $\underline{S}$. Let $Z=Z_1Z_2$.
Since $Z\in\mathcal T (\underline{S})\subseteq\mathcal T (V)$, we have $\Phi_0(Z)=Z$.
Note that $\Phi_0(Z)=P_{\mathcal H}\pi_0(Z)|_{\mathcal H}$,
for every $Z\in \mathcal C^*(\underline{S})$. Now let $Z$ be any arbitrary finite product of members from $\underline{S}$ and $\underline{S}^*$.
Since $\pi_0(\underline{S})=\underline{R}$, which is a family of normal operators, we obtain, by Fuglede-Putnam theorem that,
action of $\Phi_0$ on $Z$ has all the members from $\underline{S}^*$ at the left and all the members from $\underline{S}$ at the right.
It follows from $\text{ker}\pi=\text{ker}\Phi_0$ and $\Phi_0$ is idempotent that ker$\pi_0=\{Z-\Phi_0(Z):Z\in \mathcal C^*(\underline{S})\}$.
Also, because of the above action of $\Phi_0$, if
$Z$ is a finite product of elements from $\underline{S}$ and $\underline{S}^*$ then a simple commutator manipulation shows that $Z-\Phi_0(Z)$ belongs to the ideal generated by all the commutators $XY-YX$,
where $X,Y \in \mathcal C^*(\underline{S})\cap \mathcal T(\underline{S})$. This shows that ker$\pi_0=\mathcal I(\underline{S})$.

In order to find a completely isometric cross section, recall the completely isometric map $\rho:\pi(\mathcal C^*(\mathcal T( V)))\to \mathcal B(\mathcal H)$ defined by $Y\mapsto V^*YV$
such that $\pi\circ \rho=id_{\pi(\mathcal C^*(\mathcal T(V)))}$. Set $\rho_0:=\rho|_{\pi(\mathcal C^*(\underline{S}))}$.
Then by the definition of $\rho$ and the action of $\Phi_0$, it follows that $\rho_0(\pi(X))=V^*\pi(X)V=\Phi_0(X)\in \mathcal C^*(\underline{S})$ for all $X\in \mathcal C^*(\underline{S})$.
Thus $\text{Ran} \rho_0\subseteq \mathcal C^*(\underline{S})$ and therefore is a
completely isometric cross section.
This completes the proof of the theorem.
\end{proof}

A remark on a possible strengthening of Theorem \ref{the-iso-case} is in order.
\begin{remark}
One can work with a commuting family of $\Gamma_d$-isometries instead of just one and obtain similar results as in Thereom \ref{the-iso-case}. This was done in the pair case in \cite{B-D-S}.
\end{remark}

\section{Dual Toeplitz operators}\label{S:Dual}
Dual Toeplitz operators have been studied for a while on the Bergman space of the unit disc $\mathbb D$ in \cite{SZ2} and on the Hardy space of the Euclidean ball $\mathbb B_d$ in \cite{DE}. In our setting, consider the space
$$H^2(\mathbb G_d)^\perp = L^2(b\Gamma_d) \ominus H^2(\mathbb G_d).$$
Let $(I-Pr)$ be the projection of $L^2(b\Gamma_d)$ onto $H^2(\mathbb G_d)^\perp$. If $\varphi \in L^\infty (b\Gamma_d)$,
define the dual Toeplitz operator on $H^2(\mathbb G_d)^\perp$ by $ DT_\varphi = (I-Pr) M_\varphi|_{H^2(\mathbb G_d)^\perp}$. With respect to the decomposition above,
\begin{eqnarray}\label{matrix-of-M_phi} M_\varphi =
\left(
                 \begin{array}{cc}
                   T_\varphi & H_{\overline{\varphi}}^* \\
                   H_\varphi & DT_\varphi \\
                 \end{array}
               \right).\end{eqnarray}

\begin{lemma}
The $d$-tuple $\underline{D} = (DT_{\bar{s}_1},\dots,DT_{\bar{s}_{d-1}}, DT_\pbar)$ is a $\Gamma_d$-isometry with the $d$-tuple $(M_{\bar{s}_1},\dots,M_{\bar{s}_{d-1}}, M_\pbar)$ as its minimal $\Gamma_d$-unitary extension. \end{lemma}
\begin{proof}
It is a $\Gamma_d$-isometry because it is the restriction of the $\Gamma_d$-unitary $(M_{\bar{s}_1},\dots,M_{\bar{s}_{d-1}}, M_\pbar)$ to the space $H^2(\mathbb G_d)^\perp$.
And this extension is minimal because $M_\pbar$ is the minimal unitary extension of $DT_\pbar$.
\end{proof}
\begin{thm} A bounded operator $T$ on $H^2(\mathbb G_d)^\perp$ is a dual Toeplitz operator if and only if it satisfies the Brown--Halmos relations with respect to $\underline{D}$. \end{thm}
\begin{proof}
The easier part is showing that every dual Toeplitz operator on $H^2(\mathbb G_d)^\perp$ satisfies the Brown--Halmos relations with respect to $\underline{D}$. It follows from the following identities
$${M_{\bar{s}_i}}^*M_\varphi M_\pbar=M_\varphi M_{\bar{s}_{d-i}} \text{ and }{M_\pbar}^*M_\varphi M_\pbar=M_\varphi \text{ for every }\varphi \in L^\infty (b\Gamma_d)$$ and from the $2\times 2$ matrix representations of the operators in concern. For the converse, let $T$ on $H^2(\mathbb G_d)^\perp$ satisfy the Brown--Halmos relations with respect to the $\Gamma_d$-isometry $\underline{D}$. By part (2) of Theorem \ref{the-iso-case} and the fact that any operator that commutes with each of $M_{s_1},\dots,M_{s_{d-1}}$ and $M_p$ is of the form $M_\varphi$, for some $\varphi\in L^\infty(b\Gamma_d)$,
there is a $\varphi \in L^\infty (b\Gamma_d)$ such that $T$ is the compression of $M_{\varphi}$ to $H^2(\mathbb G_d)^\perp$.
\end{proof}

%

\vspace{0.1in} \noindent\textbf{Acknowledgement:}
 The research works of the first and second named authors are supported by DST-INSPIRE Faculty Fellowships DST/INSPIRE/04/2015/001094 and DST/INSPIRE/04/2018/002458 respectively. The second named author thanks Indian Institute of Technology, Bombay for a post-doctoral fellowship under which most of this work was done.


\begin{thebibliography}{99}

\bibitem{ALY-MAMS19} J. Agler, Z. Lykova and N. J. Young, {\em Geodesics, retracts, and the norm-preserving extension property in the symmetrized bidisc}, Mem. Amer. Math. Soc. 258 (2019), no. 1242, vii+108 pp. ISBN: 978-1-4704-3549-3.

%
%
%

\bibitem{AY_JGA04} J. Agler and N. J. Young, {\em{The hyperbolic geometry of the symmetrized bidisc}}, J. Geom. Anal. 14
(2004), 375-403.


\bibitem{AY-JOT03} J. Agler and N. J. Young, {\em{A model theory for $\Gamma$-contractions}}, J. Operator Theory 49 (2003), 45-60.

\bibitem{AC}P. Aherna and  Z. Cuckovic, {\em A Theorem of
Brown-Halmos Type for Bergman Space Toeplitz Operators},
J. Funct. Anal. 187 (2001), 200-210.



\bibitem{AW} H. Alexander and J. Wermer, {\em Several complex variables and Banach algebras}, Springer, New York, 1998.

\bibitem{A}S. Axler, {\em{Toeplitz Operators. A glimpse at Hilbert space operators,}} Oper. Theory Adv. Appl. 207 (2010), 125-133.

\bibitem{ACM}S. Axler, J. B. Conway and G. McDonald, {\em Toeplitz operators on Bergman spaces}, Canadian J. Math. 34 (1982), 466-483.


\bibitem{B-D-S} T. Bhattacharyya, B. K. Das and H. Sau, {\em Toeplitz operators on the symmetrized bidisc}, Int. Math. Res. Not. IMRN (to appear), DOI:10.1093/imrn/rnz333, available at arXiv:1706.03463.



\bibitem{BS-JFA18} T. Bhattacharyya and H. Sau, {\em{Holomorphic functions on the symmetrized bidisk - realization, interpolation and Extension}}, J. Funct. Anal. 274 (2018), no. 2, 504-524.


\bibitem{SS} S. Biswas and S. Shyam Roy, {\em Functional models of $\Gamma_d$-contractions and characterizations of $\Gamma_d$-isometries}, {J. Funct. Anal.} 266 (2014), 6224-6255.

\bibitem{BS} A. B\"ottcher, and B. Silbermann, {\em{Analysis of Toeplitz operators}}, Springer, Berlin, 1990.

\bibitem{BH} A. Brown and P. R. Halmos, {\em{Algebraic properties of Toeplitz operators}}, J. Reine Angew. Math. 213 (1963), 89-102.

\bibitem{CYU} G. Cho and Y. Yuan, {\em Bergman metric on the symmetrized bidisc and its consequences}, arXiv:2004.04637. 

\bibitem{CE} M. D. Choi and E. G. Effros, {\em {Injectivity and operator spaces}}, J. Funct. Anal. 24 (1977), 156-209.




\bibitem{DE} M. Didas and J. Eschmeier, {\em{Dual Toeplitz operators on the sphere via sperical isometries}}, Integral Equations and Operator Theory 83 (2015), 291-300.


\bibitem{DP} G. E.~Dullerud and F.~Paganini, {\em A Course in Robust Control
  Theory: A Convex Approach}, Texts in Applied Mathematics Vol.~{\bf 36},
  Springer-Verlag, New York, 2000.

\bibitem{EZ}A. Edigarian and W. Zwonek, {\em Geometry of the symmetrized polydisc}, Arch.\ Math.\ (Basel) 84 (2005), 364-374.

\bibitem{Feintuch}A. Feintuch, {\em{On asymptotic Toeplitz and Hankel operators}}, In
The Gohberg anniversary collection, Vol.II (Calgary, AB, 1988), volume 41 of Oper. Theory Adv. Appl., pages 241-254. Birkhauser, Basel, 1989.

\bibitem{GS} S. Gorai and J. Sarkar, {\em Characterizations of symmetrized polydisc}, Indian J.\ Pure Appl.\ Math.\  47 (2016), 391-397.

\bibitem{Jewell}N. P. Jewell, {\em Toeplitz operators on the Bergman spaces and in several complex variables}, Proc. London Math. Soc. 41 (1980), 193-216.




\bibitem{MSS} A. Maji, J. Sarkar and S. Sarkar, {\em{Toeplitz and asymptotic Toeplitz operators on $H^2(\mathbb D^d)$}}, Bull. Sci. Math. 146 (2018), 33-49.


\bibitem{MSRZ} G. Misra, S. Shyam Roy and G. Zhang, {\em{Reproducing kernel for a class of weighted Bergman spaces
on the symmetrized polydisc}}, Proc. Amer. Math. Soc. 141 (2013), 2361-2370.

\bibitem{Niko} N. Nikolov, {\em The symmetrized polydisc cannot be exhausted by domains biholomorphic to convex domains}, Ann.\ Polon.\ Math.\ 88 (2006), 279-283.

\bibitem{NPTZ}N. Nikolov, P. Pflug, P. J. Thomas and W. Zwonek, {\em Estimates of the Carath\'eodory metric on the symmetrized polydisc}, J. Math.\ Anal.\ Appl. 341 (2008), 140-148.

\bibitem{NPZ} N. Nikolov, P. Pflug and W. Zwonek, {\em The Lempert function of the symmetrized polydisc in higher dimensions is not a distance}, Proc.\ Amer.\ Math.\ Soc.\ 135 (2007), 2921-2928.

\bibitem{NTT}N. Nikolov, P. J. Thomas and D. Tran, {\em Lifting maps from the symmetrized polydisc in small dimensions}, Complex Anal.\ Oper.\ Theory 10 (2016), 921-941.

\bibitem{DJO} D. J. Ogle, {\em Operator and Function Theory of the Symmetrized Polydisc}, PhD Thesis, University of Newcastle upon Tyne, October 1999. 

\bibitem{Prunaru} B. Prunaru, {\em{Some exact sequences for Toeplitz algebras of spherical isometries}}, Proc. Amer. Math. Soc. 135 (2007), 3621-3630.

\bibitem{SZ1}K. Stroethoff and D. C. Zheng, {\em Toeplitz and Hankel operators on Bergman spaces}, Trans. Amer. Math. Soc. 329 (1992), 773-794.

\bibitem{SZ2}  K. Stroethoff and D. Zheng, {\em{Algebraic and spectral properties of dual Toeplitz operators}}, Trans. Amer. Math. Soc. 354 (2002), no. 6, 2495-2520.

\bibitem{TTZ}P. J. Thomas, N.V. Trao and W. Zwonek, {\em Green functions of the spectral ball and symmetrized polydisk}, J. Math. Anal. Appl. 377 (2011), 624-630.

\bibitem{TRY}
M. Trybula, {\em Invariant metrics on the symmetrized bidisc}, Complex Var. Elliptic Equ. 60 (2015), 559-565. 



%













\bibitem{Vasi}N. L. Vasilevski, {\em Commutative algebras of Toeplitz operators on the Bergman space}, Operator Theory: Advances and Applications, 185.
Birkh\"auser Verlag, Basel, 2008. xxx+417 pp.

\bibitem{Vuko}D. Vukoti\'c, {\em Analytic Toeplitz operators on the Hardy space $H^p$: a survey}, Bull. Belg. Math. Soc. Simon Stevin 10 (2003), no. 1, 101-113.

\end{thebibliography}
\end{document}